\documentclass[12pt]{article}

\usepackage{authblk}
\usepackage{geometry}
\usepackage{subcaption}
\usepackage{graphicx}
\usepackage{epstopdf}
\usepackage{xcolor}

\title{The Phase Transition in Online PCA Depends on $n/d\log(d)$, not $n/d$}
\author{
    Apratim Dey\thanks{Email: \texttt{dey@wustl.edu}}
}
\affil{\textit{Washington University in St. Louis}}

\usepackage{amsthm,amsmath,amsfonts,bm}

\newtheorem{theorem}{Theorem}[section]  
\newtheorem{lemma}[theorem]{Lemma}      

\usepackage{algorithm}
\usepackage{algpseudocode} 
\usepackage{xcolor}

\newcommand{\BBP}{\text{BBP}}
\newcommand{\PCA}{\text{PCA}}

\newcommand{\sph}{\text{sph}}

\def\eqref#1{equation~\ref{#1}}

\def\1{\bm{1}}

\DeclareMathAlphabet{\mathsfit}{\encodingdefault}{\sfdefault}{m}{sl}
\SetMathAlphabet{\mathsfit}{bold}{\encodingdefault}{\sfdefault}{bx}{n}

\def\gF{{\mathcal{F}}}

\def\gN{{\mathcal{N}}}

\def\sP{{\mathbb{P}}}

\def\sR{{\mathbb{R}}}
\def\sS{{\mathbb{S}}}

\newcommand{\E}{\mathbb{E}}

\DeclareMathOperator{\sign}{sign}

\usepackage{hyperref}
\usepackage{url}
\usepackage{graphicx}
\usepackage{placeins}
\usepackage{float}
\usepackage{natbib}

\begin{document}

\maketitle

\begin{abstract}
High dimensional statistical theory has established the importance of constant aspect ratio, when the number of dimensions ($d$) and samples ($n$) satisfy $n,d\to\infty$ with $n/d\to \gamma\in(0,\infty)$, in understanding the limits of canonical estimation problems. In particular, for estimating the top eigenvector of a $d\times d$ population covariance matrix from $n$ iid samples, the BBP phase transition gives a precise threshold -- a simple functional of the aspect ratio -- such that the top sample principal component attains nonzero asymptotic correlation with the truth only when the leading population eigenvalue exceeds it. In this paper, we show that for online / streaming algorithms the story is very different, and constant aspect ratio is insufficient for nonzero overlap. We study Oja's algorithm, the most popular method for online PCA. Let $\Sigma=\theta^2 v_0v_0^\top+I\in\mathbb{R}^{d\times d}$, and run Oja's algorithm with step size $\delta/d$ on $n$ iid samples $X_k\sim\mathcal{N}(0,\Sigma)$, with output $\hat v_n$. Then, as $n,d\to\infty$ with $n/d\log d\to\gamma\in(0,\infty)$, we establish a phase transition: $|\langle\hat v_n,v_0\rangle|\to 0$ when $\gamma<\gamma_*$, and $\to\rho_*$ when $\gamma>\gamma_*$. Here $\rho_*=\rho_*(\theta,\delta)=\sqrt{(\theta^2-\delta/2)_+/\theta^2(1+\delta/2)}$ and $\gamma_*=\gamma_*(\theta,\delta)=1/\delta(2\theta^2-\delta)_+$. Further, at criticality, when $n=[\gamma_*d\log d+\eta d]$ and $d\to\infty$, $\eta\in\mathbb{R}$, the correlation is random: $|\langle\hat v_n,v_0\rangle|\stackrel{w}{\to}\rho_*|G|\exp(\eta/2\gamma_*)/\sqrt{\rho_*^4+G^2\exp(\eta/\gamma_*)}$ where $G\sim\mathcal{N}(0,1)$. This is in stark contrast to ordinary high dimensional PCA, where nonzero overlap is possible at constant $n/d$ and improves as $n/d$ increases.

\end{abstract}
\section{Introduction}\label{sec:intro}

 We consider the canonical problem of estimating the top eigenvector $v_0\in\sS^{d-1}$ of a $d\times d$ covariance matrix $\Sigma$ given iid data $X_k\sim\gN(0,\Sigma)$, $k\geq 1$. Here, $\sS^{d-1}$ denotes the $d-$dimensional sphere $\{x\in\sR^d:\|x\|_2=1\}$, and $\gN(\mu,\Sigma)$ denotes the $d-$dimensional Gaussian distribution with mean $\mu$ and covariance matrix $\Sigma$. This is a problem with decades of cultural and intellectual significance, with important applications in statistics, signal processing, genetics, etc. 

 \medskip
 
Perhaps the most standard approach for this problem is Principal Component Analysis (PCA). Given samples $X_1,\cdots, X_n$, one computes the sample covariance matrix $S_n:=\sum_{k=1}^n X_kX_k^\top/n$ and uses the sample top principal component $\hat v_n^\PCA$ as an estimator of $v_0$. For simplicity, we ignore centering as we assume each $X_k$ has mean zero. Thanks to the works of many researchers over the last few decades (see Section \ref{sec:related_works} for an account), we have a fairly complete understanding of the statistical properties of PCA today.

 \medskip
 
 By definition, $\hat v_n^\PCA$ is the leading eigenvector of $S_n$, and computing $\hat v_n^\PCA$ requires one to store and perform computations on the full matrix $S_n$. When $d$ is very large, as in many ``big data" problems such as in genomics, computing $\hat v_n^\PCA$ quickly becomes computationally prohibitive. Consequently, online algorithms that treat each $X_k$ sequentially and only updating a running estimator $\hat v_n$ of $v_0$ are quite attractive.

\medskip
 
 Perhaps the most popular online PCA algorithm is due to Oja \citep{oja1982simplified}, to be henceforth called Oja's algorithm. One starts with an initial unit vector $\hat v_0\in \sS^{d-1}$ and proceeds iteratively as follows:
\begin{align}\label{eqn:oja_main_recursion}
\begin{split}
    \tilde v_k &= \tilde v_{k-1} + (\delta/d)\langle X_k,\tilde v_{k-1}\rangle X_k,\\
    \hat v_k &= \tilde v_k/\|\tilde v_k\|_2    
\end{split}
\end{align}Here, $\delta>0$ is a step size. Naturally, each iteration is extremely lightweight, involving simple inner products, and is therefore computationally very convenient.

 \medskip

But how good is the estimator $\hat v_n$? Namely, if we started Oja's algorithm from a \textit{random} initialization $\hat v_0$, what is the resulting overlap $|\langle \hat v_n,v_0\rangle|$? Unfortunately, our knowledge of the performance of $\hat v_n$ is not at the same level as that of $\hat v_n^\PCA$. This gap in understanding widens particularly in the high dimensional setting where $d$ is large since one cannot ignore the effect of $d$ on $n$ and hence on the overlap. Moreover, it is now well known that several high dimensional problems \citep{baikbenarouspeche,donoho2009observed} exhibit a so-called phase transition -- a critical value $n_*$ such that the algorithm succeeds when $n>>n_*$ and fails with $n<<n_*$. Does Oja's algorithm also exhibit a similar phase transition phenomenon, and can we identify it?

\medskip

Our purpose in this paper is to settle this question of the phase transition in Oja's algorithm in its completeness. We suppose that the true covariance matrix $\Sigma$ has the form $\Sigma=\theta^2v_0v_0^\top + I$. This is the celebrated Johnstone spiked model \citep{johnstone2001distribution} which has played a major role in the understanding of PCA in high dimensions. Under this model, there is a single principal direction $v_0\in\sS^{d-1}$ with corresponding top eigenvalue $\theta^2+1$, rest all eigenvalues being equal to $1$.

\medskip

Our main results are the following. Let $n,d\to\infty$ such that $n/d\log d\to\gamma\in(0,\infty)$. Then, there exists a critical threshold $\gamma_*\equiv \gamma_*(\theta,\delta)\in(0,\infty)$ and a number $\rho_*\in(0,1)$ such that
\begin{align}
\begin{split}
    |\langle\hat v_n,v\rangle| \stackrel{p}{\to} \begin{cases}
        0, & \gamma < \gamma_*,\\
        \rho_*, & \gamma > \gamma_*
    \end{cases}
\end{split}
\end{align}Here, $\stackrel{p}{\to}$ denotes convergence in probability. Further, when $n=[\gamma_*d\log d+\eta d]$ ($[\cdot]$ denoting the greatest integer function), then for any $\eta\in\sR$,
\begin{align}
    |\langle \hat v_n,v_0\rangle|\stackrel{w}{\to}\dfrac{\rho_*|G|\exp(\eta/2\gamma_*)}{\sqrt{\rho_*^4 + G^2\exp(\eta/\gamma_*)}}
\end{align}where $\stackrel{w}{\to}$ denotes weak convergence and $G\sim\gN(0,1)$. Thus, we identify the precise constants $\gamma_*,\rho_*$, the precise rate $d\log d$ and the precise stochastic behavior governing the performance of Oja's algorithm.

 \medskip

This points out an important distinction with PCA. In PCA, asymptotically positive overlap \textit{is} possible when $n$ is \textit{linear} in $d$. However, for Oja's algorithm, to achieve any asymptotically positive overlap, we must have $n/d\log d$ exceed a constant. This extra $\log d$ factor is non-negotiable and not due to looseness in the proof. Thus, Oja's algorithm will be asymptotically unsuccessful when $n$ is merely linear in $d$. The extra $\log(d)$ factor is related to the large amount of stochasticity in the algorithm. Indeed, the overlap takes $O(d\log(d))$ steps to rise above the initial $O(d^{-1/2})$ noise floor to a non-vanishing value. As the proofs of our theorems will clarify, once the algorithm reaches a non-vanishing overlap, it reaches the limiting overlap $\rho_*$ pretty quickly, in $\Theta(d)$ steps.

 \medskip

The serious statistical suboptimality of Oja's algorithm over offline PCA is worth noting. One should remember that $\hat v_n^\PCA$ enjoys certain optimality properties by extracting maximum information from the data $X_1,\cdots, X_n$; for example, it is the direction of maximal sample variability, and also is the direction of the optimal rank one reconstruction of the sample covariance matrix $S_n$. On the other hand, Oja's algorithm only performs one single, rather simple operation on each new data point $X_k$. Thus, it is expected that it would need many more samples than ordinary PCA to achieve a desired overlap. Our results are able to precisely quantify this.

 \medskip

Finally, \citet{arous2021online} have studied high dimensional online stochastic gradient descent (SGD) algorithms. One of their special cases involves a variant of Oja's algorithm, which we investigate in details in Section \ref{sec:BAGJ}. Interestingly, that variant, although visibly different from Oja's algorithm, turns out to exhibit the exact same phase transition phenomenon as Oja's algorithm with the exact same $\gamma_*,\rho_*$ and asymptotic distribution at criticality.
\section{Related Works}\label{sec:related_works}

\paragraph{PCA.} Principal Component Analysis (PCA, in short) is a major workhorse of modern data science, with applications virtually everywhere. Consequently, the past decades have seen extensive research on understanding the precise behavior of PCA.

To fix ideas, we consider the Johnstone spiked model setting defined in Section \ref{sec:intro} (also see \citet{johnstone2001distribution}), where the true covariance matrix $\Sigma=\theta^2v_0v_0^\top+I$. Recall that $\hat v_n^\PCA$ is the sample top principal component used to estimate the truth $v_0$. When $d$ is fixed and $n\to\infty$, the usual law of large numbers arguments establishes consistency: $|\langle \hat v_n^\PCA,v_0\rangle|\to 1$.

However, when one considers the high dimensional setting where $n,d$ are both large, things change substantially, and this automatic consistency of the sample principal components is invalid. In fact, in the so-called \textit{proportional asymptotics} regime where $n,d\to\infty$ with $n/d\to \gamma\in (0,\infty)$, the inverse aspect ratio $\gamma$ turns out to be a critical quantity governing the performance of PCA. In fact, there exists a critical threshold $\theta_{\text{BBP}}\equiv\theta_\BBP(\gamma)= 1/\gamma^{1/4}$ \citep{baikbenarouspeche,benaych2011eigenvalues} known as the BBP threshold, such that
\begin{align}\label{eqn:bbp}
    |\langle \hat v_n^\PCA,v_0\rangle|\to \begin{cases}
        0, & \theta < \theta_{\text{BBP}}, \\
        \sqrt{1 - \dfrac{1+\theta^2}{\theta^2(\gamma\theta^2+1)}}, & \theta > \theta_{\text{BBP}}
    \end{cases}
\end{align}In other words, to get positive asymptotic overlap, it is necessary that $\theta>\theta_{\text{BBP}}$. Extensive research over the last couple decades \citep{johnstone2001distribution,baikbenarouspeche,benaych2011eigenvalues,benaych2012singular,johnstone2018pca,perry2018optimality} has made our understanding of PCA fairly comprehensive. It is known that the BBP threshold is information theoretically optimal \citep{perry2018optimality} when $v_0$ is generic, and thus no procedure exists that can outperform PCA in the context of detection. When $v_0$ is more structured, e.g. is sparse, other algorithms have been developed to outperform PCA \citep{johnstone2009consistency,amini2008high,deshpande2016sparse,deshpande2014information}. More recently, approximate message passing algorithms have been used to outperform PCA once the signal exceeds the BBP threshold \citep{MontanariVenkataraman,LiFanWei}.

\paragraph{Online PCA.} Perhaps the simplest and most popular algorithm for online / streaming PCA is due to Oja \citep{oja1982simplified,oja1985stochastic}; also see \citet{krasulina1970method} for a similar method. Oja's algorithm has been presented in (\ref{eqn:oja_main_recursion}). A line of research has provided theoretical guarantees on online PCA algorithms; see \citet{warmuth2008randomized,boutsidis2014online,jain2016streaming,nie2016online,henriksen2019adaoja} for a non-exhaustive list. We refer the interested reader to \citet{cardot2018online} for a nice survey on the different algorithms for online PCA. Specifically related to Oja's algorithm, inspiring recent works involve \cite{lunde2021bootstrapping,kumar2023streaming,kumar2025beyond,pham2025time} among others. However, almost all these studies provide rates and upper bounds on the estimation error. In high dimensional problems, it is known that standard concentration-based results are usually unable to capture the precise statistical performance of estimators as exact constants become critical. When $n$ and $d$ are both large and neither can be neglected, the precise trajectory of Oja's algorithm therefore remains unknown. 

Having said this, there do exist results describing the precise performance of Oja's algorithm in high dimensions, particularly due to \citet{wang2016online,wang2017scaling}. However, they assume that the initialization is informative, i.e. the algorithm is initialized with $\hat v_0$ satisfying $\lim_{d\to\infty}|\langle \hat v_0,v_0\rangle|>0$. They show that given this ``warm" start, the algorithm can reach stable equilibrium with $n=\Theta(d)$ samples. Of course, it is not at all clear how one can start with an informative initialization. In contrast, our results provide the complete description of Oja's algorithm starting with random initialization. Moreover, as our Theorem \ref{thm:critical} establishes, for the entire $O(d)$ window at criticality, the transition of the limiting overlap from $0$ to $\rho_*$ is actually through a random path defined by the Gaussian variable $G$. This is in contrast to what one expects from informative initialization. 

Finally, we note that \citet{li2017diffusion} discusses a three-phase behavior of Oja's algorithm but in a low-dimensional setting. This low dimensionality has an important impact on the resulting formulae and do not correspond to genuine high dimensional Oja.

\paragraph{High dimensional SGD.} Oja's algorithm can be viewed as stochastic gradient descent (SGD) on a suitable loss function; see Section \ref{sec:BAGJ} for a discussion. A recent line of work has studied high dimensional SGD algorithms \citep{arous2021online,ben2022high,gheissari2025universality}. In particular, \citet{arous2021online} has studied as a special case a variant of Oja's algorithm (details to be found in Section \ref{sec:BAGJ}) and derived the $d\log d$ rate for that algorithm (not for Oja) as related to the information exponent of the problem. However, they provide order estimates, whereas we are able to completely track down the relevant constants and precise behavior. Moreover, we show in Theorem \ref{thm:BAGJ} that both Oja and their algorithm have the exact same phase transition behavior, with exactly the same constants and asymptotic distributions. \citet{ben2022high} provide differential equations describing summary statistics of high dimensional SGD updates. Their setup is different from ours, since Oja involves an additional normalization step after the loss function update, which has a non-trivial effect on the final dynamics.

\paragraph{Notations.} For a vector $x$, $\|x\|$ will typically denote the Euclidean norm, unless otherwise specified. $\sS^{d-1}:=\{x\in\sR^d:\|x\|=1\}$ denotes the $d-$dimensional sphere. $\gN(\mu,\Sigma)$ denotes the Gaussian distribution with mean $\mu$ and covariance matrix $\Sigma$. $\chi^2_k$ denotes the chi-squared random variable with $k$ degrees of freedom. Unless otherwise stated, $\log(\cdot)$ will always represent natural logarithm. For a set $S$, $Unif(S)$ will denote the uniform distribution over $S$. 
A sequence $X_n$ of random variables converges in probability to a random variable $X$, denoted by $X_n\stackrel{p}{\to}X$ if for any $\epsilon>0$, $\sP(|X_n-X|>\epsilon)\to 0$ as $n\to\infty$. $X_n$ converges weakly to $X$, denotes by $X_n\stackrel{w}{\to}X$ if for every $x$ in the continuity set of the cdf $F_X$ of $X$, $P(X_n\leq x)\to P(X\leq x)$ as $n\to\infty$. See any standard text on probability theory for more details, for example \citet{billingsley2017probability}.
Finally, $C$ will denote unspecified but universal constants, that will change from place to place. We will not track these constants explicitly.

\section{Main Results}\label{sec:main_results}

As described in Section \ref{sec:intro}, we set ourselves in the Johnstone spiked model \citep{johnstone2001distribution} and assume that the true covariance matrix $\Sigma$ has the form
\begin{align*}
    \Sigma = \theta^2 v_0v_0^\top + I
\end{align*}Here, $\theta>0$ and $v_0\in\sS^{d-1}$ are both unknown. The goal is to estimate $v_0$, the top population principal component, using iid data $X_k\sim\gN(0,\Sigma)$ arriving in a streaming / online fashion, when $d$ is large.

\subsection{Main Recursion}

Recall Oja's iterative algorithm from (\ref{eqn:oja_main_recursion}). Starting with $\hat v_0\in\sS^{d-1}$ and a step size $\delta>0$, Oja's algorithm iteratively produces updates
\begin{align}
\begin{split}
    \tilde v_k &= \tilde v_{k-1} + (\delta/d)\langle X_k,\tilde v_{k-1}\rangle X_k,\\
    \hat v_k &= \tilde v_k/\|\tilde v_k\|    
\end{split}
\end{align}Note that the step size is actually $\delta/d$, although in whatever follows, we might abuse notation and call $\delta$ the step size. Observe that the step size is constant across all iterations, although vanishingly small as $d$ is large. We adopt this ``constant" step size following prior high dimensional analyses of online algorithms \citep{wang2016online,wang2017scaling,arous2021online,ben2022high}.

\medskip

Taking inner product with $v_0$, and setting $\rho_k:=\langle v_k,v_0\rangle$ as the (signed) overlap at iteration $k$, we get
\begin{align}\label{eq:recursion_rho}
    \rho_k &= \dfrac{\rho_{k-1}+(\delta/d)\langle X_k,v\rangle \langle X_k,\hat v_{k-1}\rangle}{\|\hat v_{k-1}+(\delta/d)\langle X_k,\hat v_{k-1}\rangle X_k\|}
\end{align}The denominator can be further simplified. Towards this, define
\begin{align*}
    A_k = \langle X_k,v_0\rangle, \quad B_k =\langle X_k,\hat v_{k-1}\rangle
\end{align*}Then, a simple calculation yields
\begin{align}\label{eqn:recursion_ABC}
    \rho_k &= \dfrac{\rho_{k-1} + (\delta/d)A_kB_k}{[1 + (2\delta/d)B_k^2 + (\delta/d)^2B_k^2\|X_k\|^2]^{1/2}}
\end{align}
It would be helpful to introduce some more random variables. Since $\langle \hat v_{k-1},v_0\rangle=\rho_{k-1}$, and both $\hat v_{k-1},v_0$ are unit vectors, we may write
\begin{align*}
    \hat v_{k-1} &= \rho_{k-1}v_0 + \sqrt{1-\rho^2_{k-1}}e_{k-1}
\end{align*}where $e_{k-1}\in\sS^{d-1}$ is orthogonal to $v_0$: $\langle e_{k-1}, v_0\rangle = 0$. Then, define
\begin{align*}
    C_k = \langle X_k,e_{k-1}\rangle
\end{align*}Then, we get the following representation:
\begin{align*}
    B_k &= \rho_{k-1}A_k + \sqrt{1-\rho^2_{k-1}}C_k
\end{align*}
 \medskip
Since $\{v_0,e_{k-1}\}$ is a set of orthonormal vectors, we can decompose $X_k$ as
\begin{align*}
    X_k &= \langle X_k,v_0\rangle v_0 + \langle X_k, e_{k-1}\rangle e_{k-1} + D_k
\end{align*}where $D_k$ is the projection of $X_k$ orthogonal to $\text{Span}(\{v_0,e_{k-1}\})$. That is, letting $P_{k-1}\equiv P_{\{v_0,e_{k-1}\}}=v_0v_0^\top + e_{k-1}e_{k-1}^\top$ denote the projection matrix onto $\text{Span}(\{v_0,e_{k-1}\})$, and denoting by $P_{k-1}^\perp=I-P_{k-1}$ the projection matrix onto the ortho-complement of $\text{Span}(\{v_0,e_{k-1}\})$, we have $D_k=P_{k-1}^\perp X_k$. We then have the following lemma.

\begin{lemma}\label{lemma:Xk_components_indep}
The random variables $A_k, C_k, \|D_k\|^2$ are mutually independent, and also independent of $\gF_{k-1}$. Further,
\begin{align*}
    A_k\sim \gN(0,\theta^2+1), \quad C_k\sim \gN(0,1),\quad \|D_k\|^2\sim \chi^2_{d-2}
\end{align*}
\end{lemma}

 \medskip

Lemma \ref{lemma:Xk_components_indep} will be remarkably helpful in computing expectations of relevant quantities downstream. Now, we present our first important lemma that takes advantage of the high dimensionality i.e. that $d$ is large, and simplifies recursion (\ref{eqn:recursion_ABC}).

\begin{lemma}\label{lemma:main_recursion}
Define $\alpha=\delta(\theta^2-\delta/2)$ and $\beta=\delta\theta^2(1+\delta/2)$. Then, recursion (\ref{eqn:recursion_ABC}) can be re-written as
    \begin{align}\label{eqn:recursion_final}
      \rho_k &= (1 + \alpha/d)\rho_{k-1} - \beta\rho^3_{k-1}/d + M_k/d + R_k/d^2
    \end{align}
    Here, $M_k$ is a martingale difference sequence adapted to the filtration $\gF_k$, so that $\E[M_k|\gF_{k-1}]=0$ for all $k$. Further, $M_k$ and $R_k$ have bounded moments of all orders, that is, for any $j\geq 1$, there is a constant $C_j>0$ such that
    \begin{align*}
        \sup_{k\geq 1}(\E|M_k|^j + \E|R_k|^j) &\leq C_j
    \end{align*}Consequently, setting $c_d\equiv 1 + \alpha/d$, for any $m< n$, we get\begin{align}\label{eqn:recursion_rho_unfolded_final}
        \rho_n &= c_d^{n-m}\rho_m - \dfrac{\beta}{d}\sum_{k=m+1}^nc_d^{n-k}\rho^3_{k-1} + \dfrac{1}{d}\sum_{k=m+1}^n c_d^{n-k}M_k + \dfrac{1}{d^2}\sum_{k=m+1}^n c_d^{n-k}R_k
    \end{align}
\end{lemma}

\medskip

Recursions (\ref{eqn:recursion_final}) and \ref{eqn:recursion_rho_unfolded_final} will play an important role in the proofs of our main theorems. Note that a priori there is no reason why $\alpha$ should be positive, since $\delta$ is the step size chosen in Oja's algorithm and $\theta$ is the true signal value in $\Sigma$ -- two completely unrelated quantities. However, the sign of $\alpha$ plays a major role in the success of Oja's algorithm. This is documented in our first theorem in the next section.

\subsection{Main Theorems}

Our first theorem establishes that if $\alpha<0$, then Oja's algorithm is asymptotically powerless no matter how large $n$ grows with $d$.

\begin{theorem}\label{thm:alpha0}
    Suppose $\alpha<0$, i.e. $\delta > 2\theta^2$. Let $\hat v_n$ be the output of Oja's algorithm after $n$ steps starting with a random initialization $\hat v_0\sim Unif(\sS^{d-1})$. Then, for any $n,d\to\infty$,
    \begin{align*}
        \langle \hat v_n,v_0\rangle \stackrel{p}{\to}0
    \end{align*}
\end{theorem}

Theorem \ref{thm:alpha0} highlights the importance of choosing a small enough step size $\delta$. We note that a similar success/failure threshold for unnormalized SGD depending on the step size was also identified in \citet{ben2022high}. Henceforth, we will assume $\alpha > 0$ i.e. $\delta < 2\theta^2$. Define 
\begin{align}\label{eqn:rho_gamma_crit_defns}
\begin{split}
   \gamma_*&\equiv \gamma_*(\theta,\delta)=\dfrac{1}{2\alpha}=\dfrac{1}{\delta(2\theta^2-\delta)} \\
   \rho_* &\equiv \rho_*(\theta,\delta) = \sqrt{\dfrac{\alpha}{\beta}} = \sqrt{\dfrac{\theta^2-\delta/2}{\theta^2(1+\delta/2)}}    
\end{split}
\end{align}
 Then, the following Theorem \ref{thm:subcritical} and Theorem \ref{thm:supercritical} together establish $\gamma_*$ as the phase transition threshold in Oja's algorithm.

\begin{theorem}\label{thm:subcritical}(Subcritical phase)
    Suppose $n/d\log d\to \gamma$ as $d\to\infty$. Then,
    \begin{align*}
        \gamma<\gamma_*\implies \rho_n\stackrel{p}{\to}0
    \end{align*}
\end{theorem}

\begin{theorem}\label{thm:supercritical}(Supercritical phase)
    Suppose $n/d\log d\to \gamma$ as $d\to\infty$. Then,
    \begin{align*}
        \gamma>\gamma_*\implies |\rho_n|\stackrel{p}{\to}\rho_*
    \end{align*}
\end{theorem}

Thus, Theorem \ref{thm:supercritical} establishes that increasing the sample size $n$ does not yield further benefit. Even when $\gamma\to\infty$, the algorithm plateaus at $\rho_*$. This is also evident in the numerical experiments we present in Section \ref{sec:experiments}.

\medskip

Finally, we present the result at criticality i.e. when $n/d\log d\to \gamma_*$. In this case, the asymptotic overlap is random and non-degenerate over a $\Theta(d)$ window.
\begin{theorem}\label{thm:critical}(Critical phase)
    Let $n=[\gamma_*d\log d+\eta d]$, $\eta\in\sR$. Then, as $d\to\infty$,
    \begin{align*}
        |\rho_n| \stackrel{w}{\to} \dfrac{\rho_*\exp(\alpha\eta)|G|}{\sqrt{\rho_*^4 + \exp(2\alpha\eta) G^2}}
    \end{align*} where $G\sim\gN(0,1)$.
\end{theorem}

Note that at criticality, the overlap interpolates between $0$ and $\rho_*$, from $\eta\to-\infty$ to $\eta\to\infty$. For any $\eta\in\sR$, the asymptotic distribution of $\rho_n$ is non-zero. Treating this as an informative initialization for subsequent iterations, $\Theta(d)$ iterations suffice for Oja's algorithm to reach the equilibrium $\rho_*$. Further, $\rho_*$ will always smaller than $1$, although can limit to $1$ when $\delta\to 0$. Thus, we recover the conclusions in \citet{wang2016online}.

\medskip

The reader might find the emergence of the gaussian variable $G$ curious. As it is evident in the proof of Theorem \ref{thm:critical} in Section \ref{sec:proofs}, $G$ has contribution both from the random initialization $\hat\rho_0$ and a random walk term while the algorithm rises from the $O(d^{-1/2})$ floor to achieve a nonzero overlap. The first time when the algorithm escapes this so-called search phase (using the terminology from \citet{arous2021online}) is random and unknown. Thus, the sign of $G$ is not predictable from merely the sign of the initial overlap.

\medskip

Together, Theorems \ref{thm:subcritical}, \ref{thm:supercritical} and \ref{thm:critical} completely characterize the phase transition in Oja's algorithm.

\subsection{Properties of $\gamma_*$ and $\rho_*$}

In this subsection, we investigate the behavior of $\gamma_*$ and $\rho_*$ as functions of $\theta,\delta$. First, we study $\gamma_*$. It is expected that as the signal strength $\theta$ increases, Oja's algorithm would spend less time searching for nonzero overlap, and hence $\gamma_*$ would decrease. The following lemma documents this.

\medskip

\begin{lemma}\label{lemma:gamma_star_behavior}
    For any $\theta>0$, $\delta\mapsto \gamma_*(\theta,\delta)$ has a unique minimum at $\delta_*=\theta^2$. Further, both functions $\theta\mapsto \gamma_*(\theta,\delta)$ for any fixed $\delta$ and $\theta\mapsto \gamma_*(\theta,\delta_*)$ are decreasing.
\end{lemma}

\begin{proof}[Proof of Lemma \ref{lemma:gamma_star_behavior}]
    Recall from (\ref{eqn:rho_gamma_crit_defns}) that $\gamma_*(\theta,\delta)=1/\delta(2\theta^2-\delta)$. The inverse function $\delta\mapsto \delta(2\theta^2-\delta)$ is a parabola opening downwards and hence has a unique maximum at $\delta_*=\theta^2$. Consequently, $\delta\mapsto\gamma_*(\theta,\delta)$ has a unique minimum at $\delta_*$.

    Next, for any fixed $\delta$, the function $\theta\mapsto \delta(2\theta^2-\delta)$ is increasing, and so $\theta\mapsto \gamma_*(\theta,\delta)$ is decreasing. Finally, $\gamma_*(\theta,\delta_*)=1/\theta^4$. Clearly, $\theta\mapsto \gamma_*(\theta,\delta_*)$ is also decreasing.
\end{proof}

\medskip

The next lemma documents the behavior of $\rho_*(\theta,\delta)$. Once again, we expect that as $\theta$ increases, $\rho_*$ would increase. However, what it also shows is that if we choose $\delta=\delta_*$ to optimize the search time, then the resulting correlation suffers, and actually gets worse as the signal strength $\theta$ increases.

\medskip

\begin{lemma}\label{lemma:rho_star_behavior}
    The function $\delta\mapsto\rho_*(\theta,\delta)$ is decreasing, for any $\theta$. It takes values $0$ at $\delta=2\theta^2$ and $1$ at $\delta=0$. Further, the function $\theta\mapsto \rho_*(\theta,\delta)$ is increasing for any $\delta$, but the function $\theta\mapsto \rho_*(\theta,\delta_*)$ is decreasing.
\end{lemma}

\begin{proof}[Proof of Lemma \ref{lemma:rho_star_behavior}]
    Recall the expression of $\rho_*$ from (\ref{eqn:rho_gamma_crit_defns}). It is clearly decreasing in $\delta$ since the numerator is decreasing and the denominator is increasing in $\delta$. Also, $\rho_*(\theta,0)=1$ and $\rho_*(\theta,2\theta^2)=0$ are checked readily.

    Next, we can re-write $\rho_*$ as
    \begin{align*}
        \rho_*(\theta,\delta) &= \sqrt{\dfrac{1 - \delta/2\theta^2}{1+\delta/2}}
    \end{align*}which immediately makes it evident that $\theta\mapsto\rho_*(\theta,\delta)$ is increasing. Finally, setting $\delta_*=\theta^2$, we get
    \begin{align*}
        \rho_*(\theta,\delta_*) &= \sqrt{\dfrac{1}{2+\theta^2}}
    \end{align*}which immediately implies $\theta\mapsto\rho_*(\theta,\delta_*)$ is decreasing.
\end{proof}

Figures \ref{fig:rho_star_plot} and \ref{fig:rho_star_at_delta_star_plot} demonstrate the behavior of $\rho_*(\theta,\delta)$.

\begin{figure}
    \centering
    \includegraphics[width=0.8\linewidth]{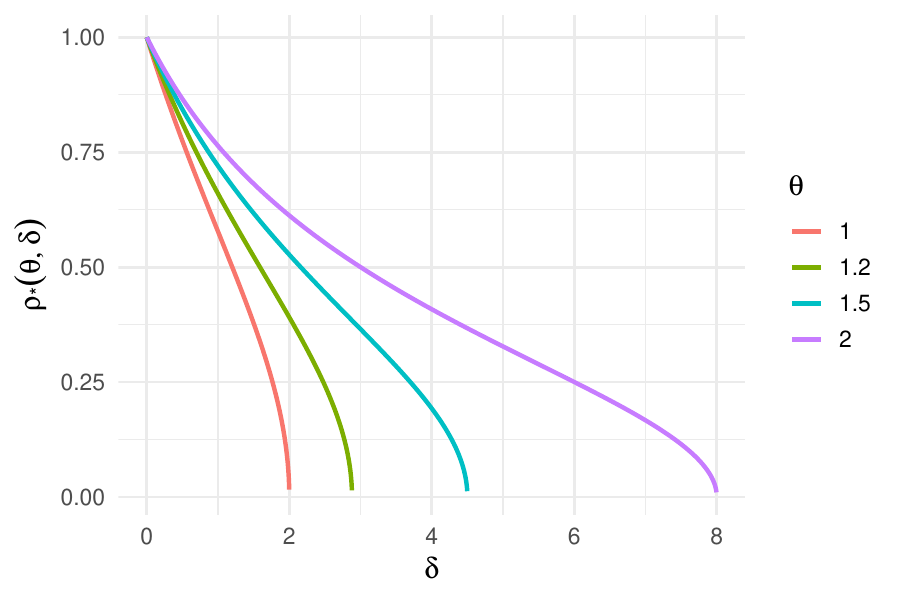}
    \caption{Plot of $\rho_*(\theta,\delta)$ as a function of $\delta$ for different values of $\theta$. When $\delta>2\theta^2$, i.e. $\alpha<0$, $\rho_*(\theta,\delta)=0$. Larger $\theta$'s show better correlation at any $\delta$.}
    \label{fig:rho_star_plot}
\end{figure}

\begin{figure}
    \centering
    \includegraphics[width=0.8\linewidth]{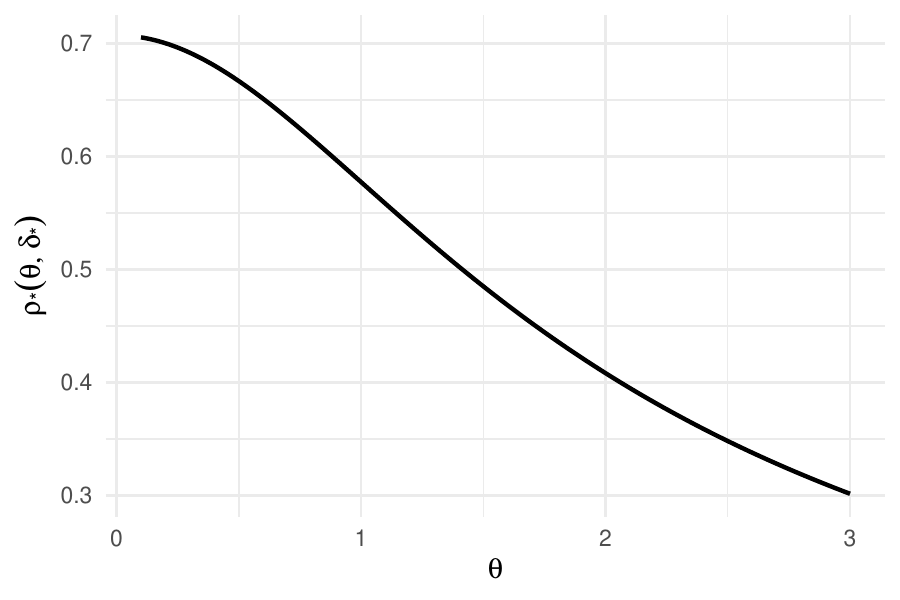}
    \caption{Plot of $\rho_*(\theta,\delta_*)$ as a function of $\theta$, where $\delta_*=\theta^2$ is the optimal $\delta$ minimizing $\gamma_*(\theta,\delta)$. The curve is decreasing in $\theta$.}
    \label{fig:rho_star_at_delta_star_plot}
\end{figure}

\subsection{Comparison with offline PCA and a curious connection}

Recall the discussion on PCA and BBP phase transition in Section \ref{sec:related_works}. For our purposes, we will write the BBP phase transition a little differently. Let
\begin{align}\label{eqn:gamma_BBP_def}
  \gamma_\BBP\equiv\gamma_\BBP(\theta)=1/\theta^4  
\end{align}
 Then, with $n/d\to\gamma\in(0,\infty)$ -- note the ratio is $n/d$ and not $n/d\log d$ -- the performance of PCA can be re-expressed as
\begin{align*}
    |\langle \hat v_n^\PCA,v_0\rangle|\to \begin{cases}
        0, & \gamma < \gamma_{\text{BBP}}, \\
        \sqrt{1 - \dfrac{1+\theta^2}{\theta^2(\gamma\theta^2+1)}}, & \gamma > \gamma_{\text{BBP}}
    \end{cases}
\end{align*}Thus, for offline PCA, as long as $\lim n/d > \gamma_*$, nonzero overlap is possible. Put another way, $n$ can be linear in $d$ and still we achieve nonzero overlap. In contrast, for Oja, Theorems \ref{thm:subcritical} and \ref{thm:supercritical} show that this is not the case, and $n$ must be of the order $d\log d$ for any nonzero overlap.

 \medskip

Of course, Oja's algorithm is online, while standard PCA is offline. Thus they are a priori unrelated. However, surprisingly, we do identify an interesting connection between the two, unraveled through the optimal $\delta_*$ derived in the previous subsection. We have seen in Lemma \ref{lemma:gamma_star_behavior} that $\delta_*=\theta^2$ minimizes $\delta\mapsto \gamma_*(\theta,\delta)$ and in that case, implies that
\begin{align}\label{eqn:gamma_star_def_optimal}
   \gamma_*(\theta,\delta_*)=1/\theta^4 
\end{align}
Note that this is exactly identical to the BBP phase transition threshold $\gamma_\BBP$ we defined in (\ref{eqn:gamma_BBP_def}). That is, formula (\ref{eqn:gamma_BBP_def}) and (\ref{eqn:gamma_star_def_optimal}) together conclude
\begin{align*}
    \gamma_\BBP = \gamma_*(\theta,\delta_*) = 1/\theta^4
\end{align*}It is indeed curious that Oja's algorithm with specific $\delta_*$ chosen to minimize the time taken to reach a nonzero overlap, and offline PCA have the exact same phase transition location for their two kinds of \textit{inverse aspect ratios}: $n/d\log d$ for Oja and $n/d$ for PCA!
\section{Experiments}\label{sec:experiments}

We present experimental evidence validating our main theorems in Section \ref{sec:main_results}. We vary $d\in\{500, 800, 1000, 2000, 5000, 10000\}$, $\theta\in\{1,2\}$ and $\delta\in\{0.1, 0.2, 0.6, 1\}$. Note that $\delta<2\theta^2$ is always satisfied. 

\medskip

For each such $(d,\theta,\delta)$ triple, we perform the following experiment $100$ times. We take the truth $v_0$ to be a random uniformly distributed vector on $\sS^{d-1}$ and generate data $X_n\sim \gN(0,\Sigma)$ with $\Sigma=\theta^2v_0v_0^\top+I$ sequentially for $1\leq n\leq n_{\max}\equiv [3\gamma_*d\log d]$, and run Oja's algorithm. We take $n_{\max}=[3\gamma_*d\log d]$ in order to completely cover the phase transition region. For each run, we record the overlap $\langle \hat v_n,v_0\rangle$ for every iteration $n$. Thus, for each triple $(d,\theta,\delta)$, we therefore have $100$ Monte Carlo trajectories of overlaps.

\medskip

Figures \ref{fig:theta_1_delta_0.1_0.2}, \ref{fig:theta_1_delta_0.6_1}, \ref{fig:theta_2_delta_0.1_0.2} and \ref{fig:theta_2_delta_0.6_1} depict the performance of Oja's algorithm. For each $(\theta,\delta, d)$ tuple, we compute the median curve over $100$ Monte Carlo overlap trajectories and plot it along with 1st and 3rd quartile bands. The x-axis shows iteration count divided by $n_*$, so that the critical phenomenon would be around $1$, which is what the red vertical line is. In all the cases, we see that the algorithm transitions from zero to non-zero overlap in a window around $n_*$. As $d$ increases, the curves get closer to the $x=1$ line, demonstrating the phase transition phenomenon. Also, we see that the curves all saturate at the theoretical value $\rho_*$, and do not improve further.

\medskip

There is one point that we would like to highlight about the figures. As $\delta$ gets large, we can see that right after $x=1$, the trajectory seems to reach overlap $\rho_*$. One might wonder why this happens. The formula makes it clear. Recall from Theorem \ref{thm:critical}, the limiting overlap is given by $$\dfrac{\rho_*|G|\exp(\eta/2\gamma_*)}{\sqrt{\rho_*^4 + G^2\exp(\eta/\gamma_*)}}$$
Fix $\eta=0$ so that we are looking exactly at the middle of the phase transition window. As $\delta$ grows, we have already seen from Lemma \ref{lemma:rho_star_behavior} that $\rho_*$ decreases. Consequently, $\rho_*^4$ may be quite small, much smaller than the gaussian variable $|G|$. This makes
\begin{align*}
    \dfrac{\rho_*|G|}{\sqrt{\rho_*^4 + G^2}}\approx \rho_*
\end{align*}Indeed, consider the subfigure in Figure \ref{fig:theta_1_delta_0.6_1} corresponding to $\theta=1,\delta=1$. Here, $\rho_*=0.58$, thus $\rho_*^4=0.11$. Standard Gaussian $G$ would often be much larger than $\rho_*^4$. This explains why we see that the algorithm apparently reaches $\rho_*$ quite rapidly after the $x=1$ line.

\medskip

Let $S=[-1,1]^2\setminus \{(x,y):|x|=|y|\}$. Define the transformation $T:S\to \sR$ by 
\begin{align*}
    T(\hat\rho,\rho)\equiv \dfrac{\hat\rho\rho^2}{\sqrt{\rho^2-\hat\rho^2}}
\end{align*}Then, 
Theorem \ref{thm:critical} implies that $|T(\hat\rho_{n_*},\rho_*)|\stackrel{w}{\to}|G|$
for $G\sim\gN(0,1)$. Actually, as is made evident through the proof of Theorem \ref{thm:critical}, it turns out that this limit holds without the absolute value as well: $T(\hat\rho_{n_*},\rho_*)\stackrel{w}{\to}G$. Thus, to verify that the limit distribution identified in Theorem \ref{thm:critical} is correct, we do the following. For each $(\theta,\delta)$ pair, we take the $100$ (signed) overlap values corresponding to the 100 Monte Carlo trajectories at $n=n_*$. Then, we apply the transformation $T$ to them and do a standard normal QQ plot. Figure \ref{fig:qq_theta_both} shows that the agreement is very good, thereby corroborating that the asymptotic distribution we found is correct.

\begin{figure}[p]
    \centering
    \vfill
    \begin{subfigure}{0.8\textwidth}
        \centering
        \includegraphics[width=\linewidth]{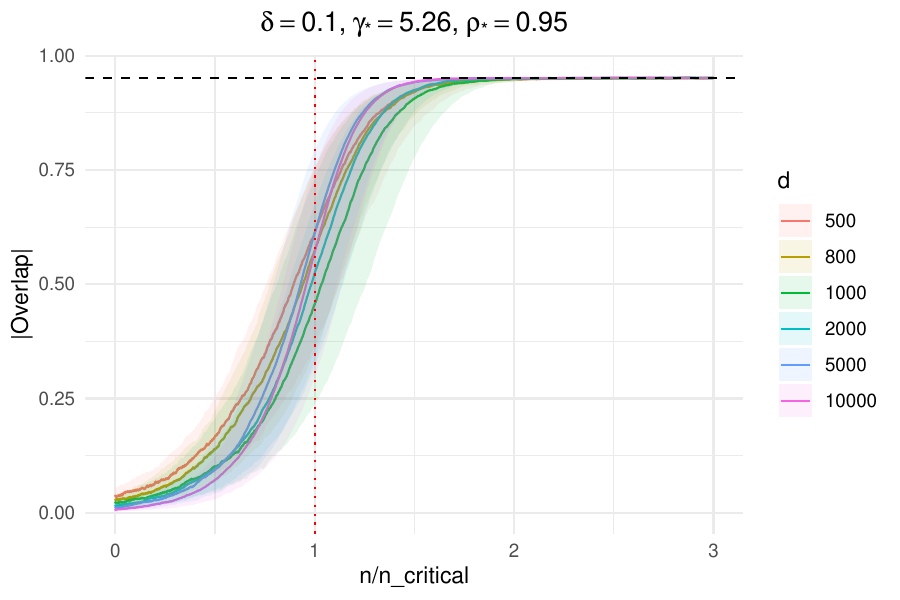}
    \end{subfigure}

    \vspace{2em}

    \begin{subfigure}{0.8\textwidth}
        \centering
        \includegraphics[width=\linewidth]{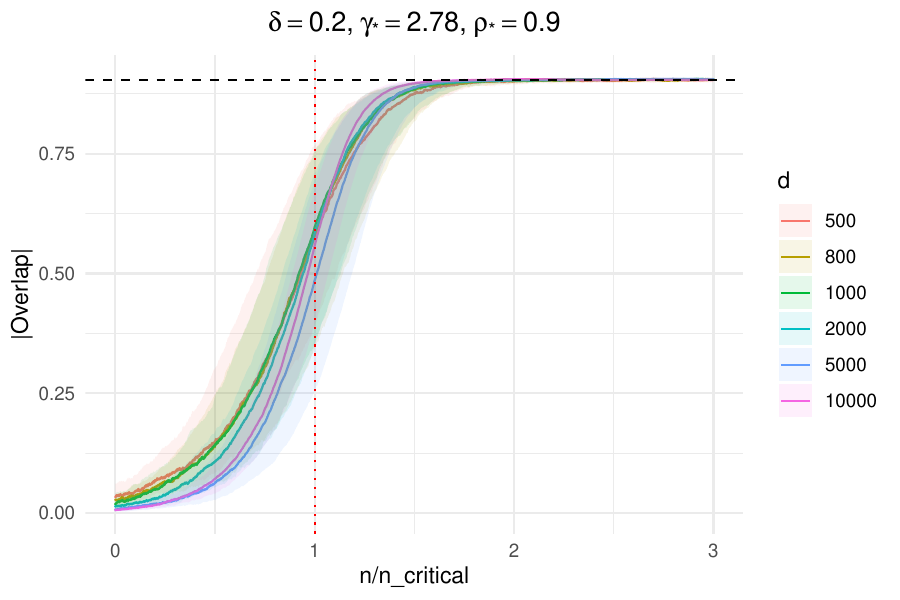}
    \end{subfigure}
    \vfill
    \caption{Performance of Oja's algorithm for $\theta=1$ and $\delta\in\{0.1, 0.2\}$ across different $d$. x-axis is rescaled to $n/n_*$ so that the vertical red dotted line shows the phase transition threshold at $1$. Horizontal dotted line shows $\rho_*$.}
    \label{fig:theta_1_delta_0.1_0.2}
\end{figure}

\begin{figure}[p]
    \centering
    \vfill
    \begin{subfigure}{0.8\textwidth}
        \centering
        \includegraphics[width=\linewidth]{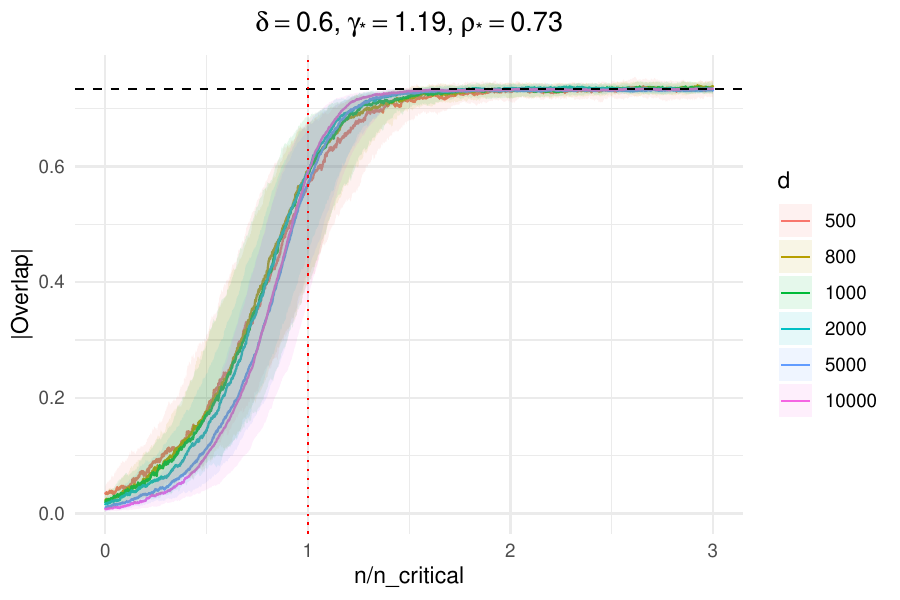}
    \end{subfigure}

    \vspace{2em}

    \begin{subfigure}{0.8\textwidth}
        \centering
        \includegraphics[width=\linewidth]{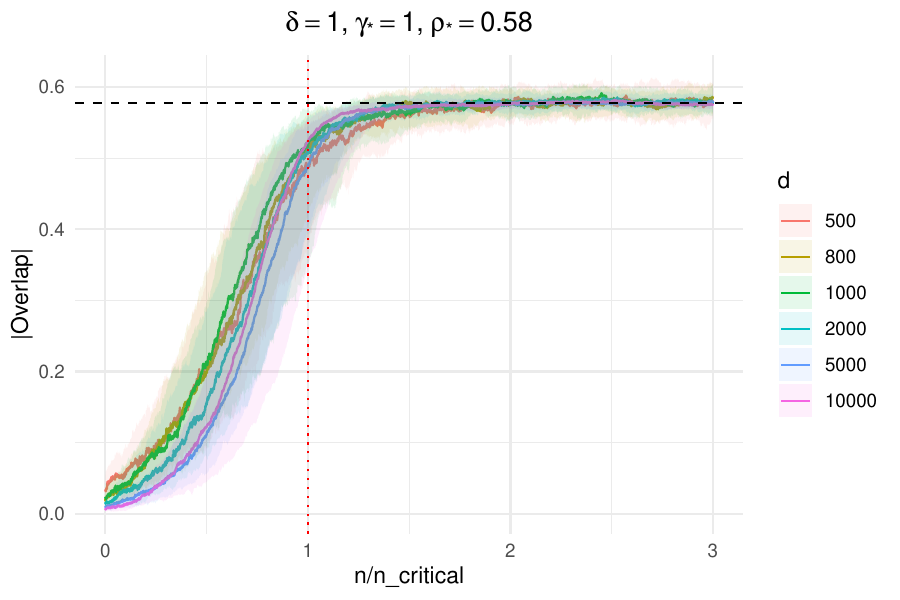}
    \end{subfigure}
    \vfill
    \caption{Performance of Oja's algorithm for $\theta=1$ and $\delta\in\{0.6, 1\}$ across different $d$. x-axis is rescaled to $n/n_*$ so that the vertical red dotted line shows the phase transition threshold at $1$. Horizontal dotted line shows $\rho_*$.}
    \label{fig:theta_1_delta_0.6_1}
\end{figure}

\begin{figure}[p]
    \centering
    \vfill
    \begin{subfigure}{0.8\textwidth}
        \centering
        \includegraphics[width=\linewidth]{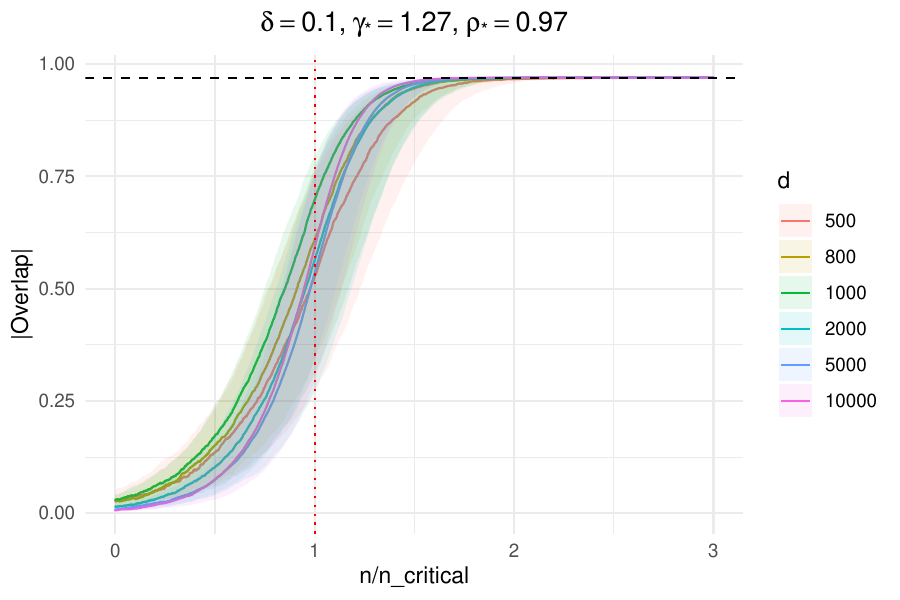}
    \end{subfigure}

    \vspace{2em}

    \begin{subfigure}{0.8\textwidth}
        \centering
        \includegraphics[width=\linewidth]{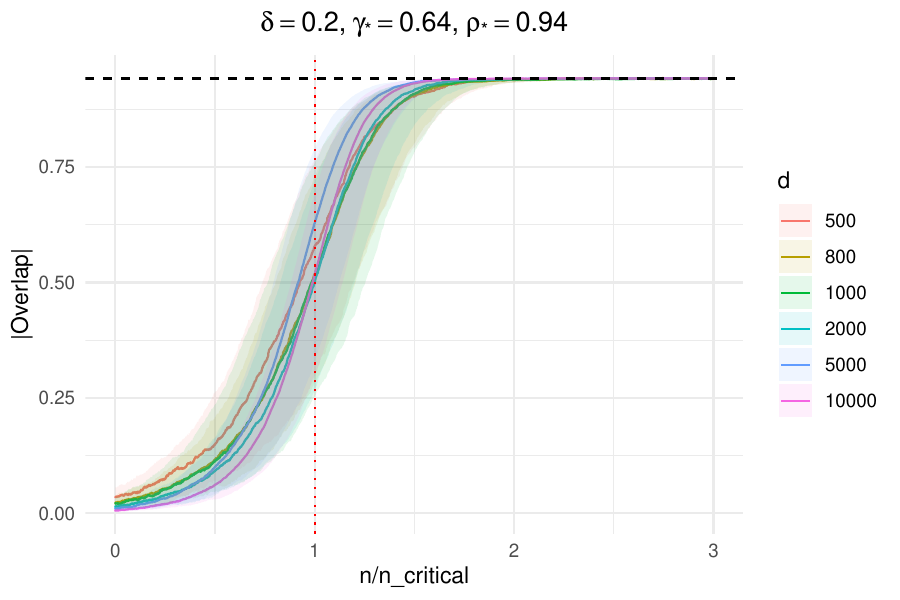}
    \end{subfigure}
    \vfill
    \caption{Performance of Oja's algorithm for $\theta=2$ and $\delta\in\{0.1, 0.2\}$ across different $d$. x-axis is rescaled to $n/n_*$ so that the vertical red dotted line shows the phase transition threshold at $1$. Horizontal dotted line shows $\rho_*$.}
    \label{fig:theta_2_delta_0.1_0.2}
\end{figure}

\begin{figure}[p]
    \centering
    \vfill
    \begin{subfigure}{0.8\textwidth}
        \centering
        \includegraphics[width=\linewidth]{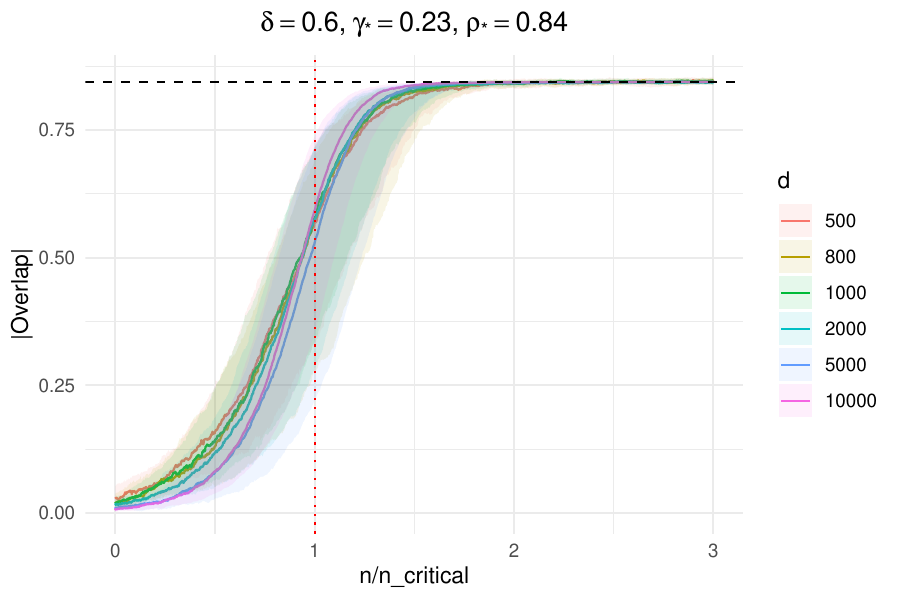}
    \end{subfigure}

    \vspace{2em}

    \begin{subfigure}{0.8\textwidth}
        \centering
        \includegraphics[width=\linewidth]{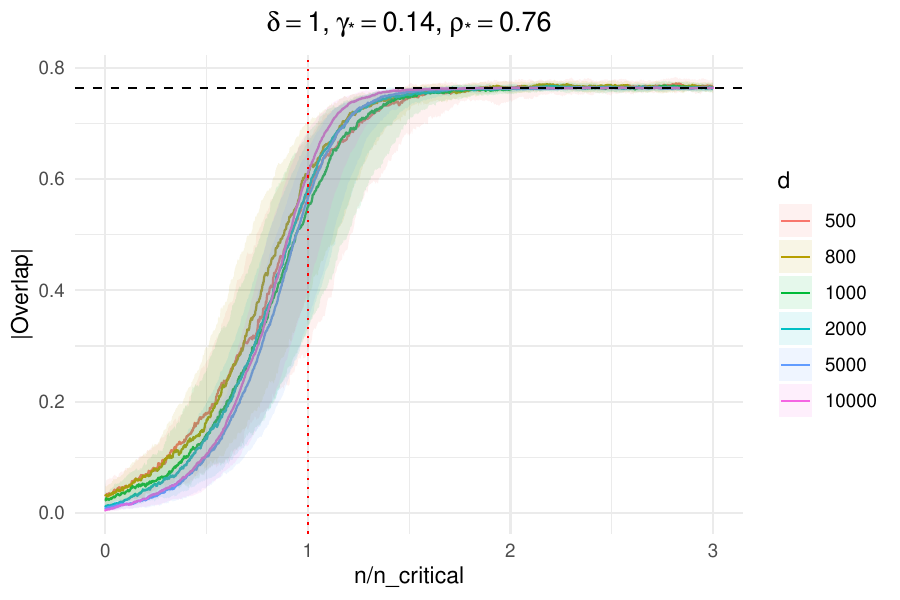}
    \end{subfigure}
    \vfill
    \caption{Performance of Oja's algorithm for $\theta=2$ and $\delta\in\{0.6, 1\}$ across different $d$. x-axis is rescaled to $n/n_*$ so that the vertical red dotted line shows the phase transition threshold at $1$. Horizontal dotted line shows $\rho_*$.}
    \label{fig:theta_2_delta_0.6_1}
\end{figure}

\begin{figure}[p]
    \centering

    \begin{subfigure}{0.36\textwidth}
        \centering
        \includegraphics[width=\linewidth]{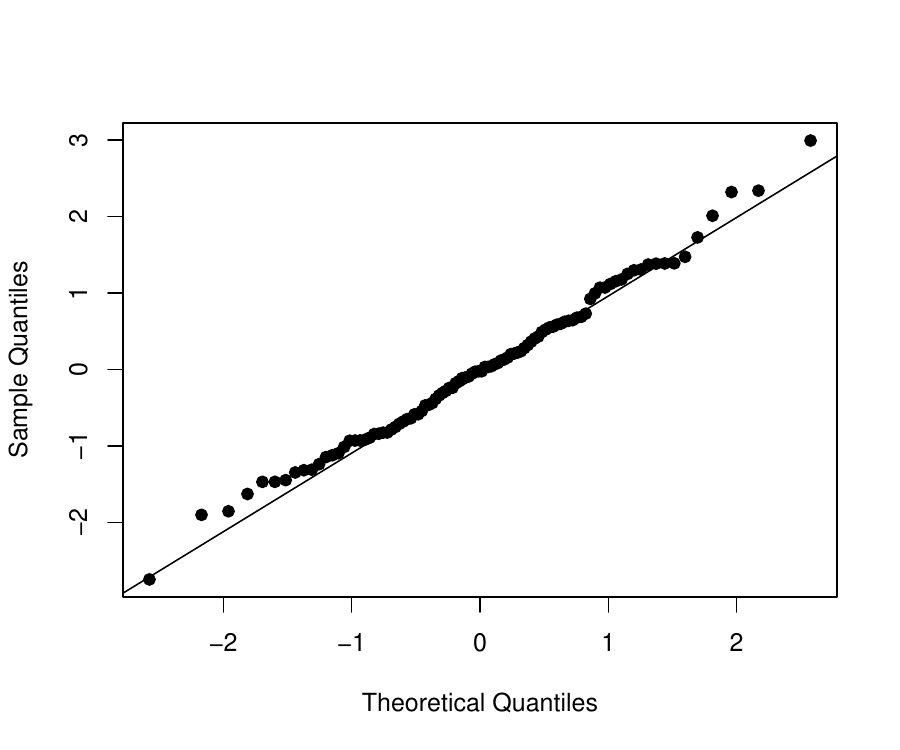}
    \end{subfigure}
    \hspace{0.02\textwidth}
    \begin{subfigure}{0.36\textwidth}
        \centering
        \includegraphics[width=\linewidth]{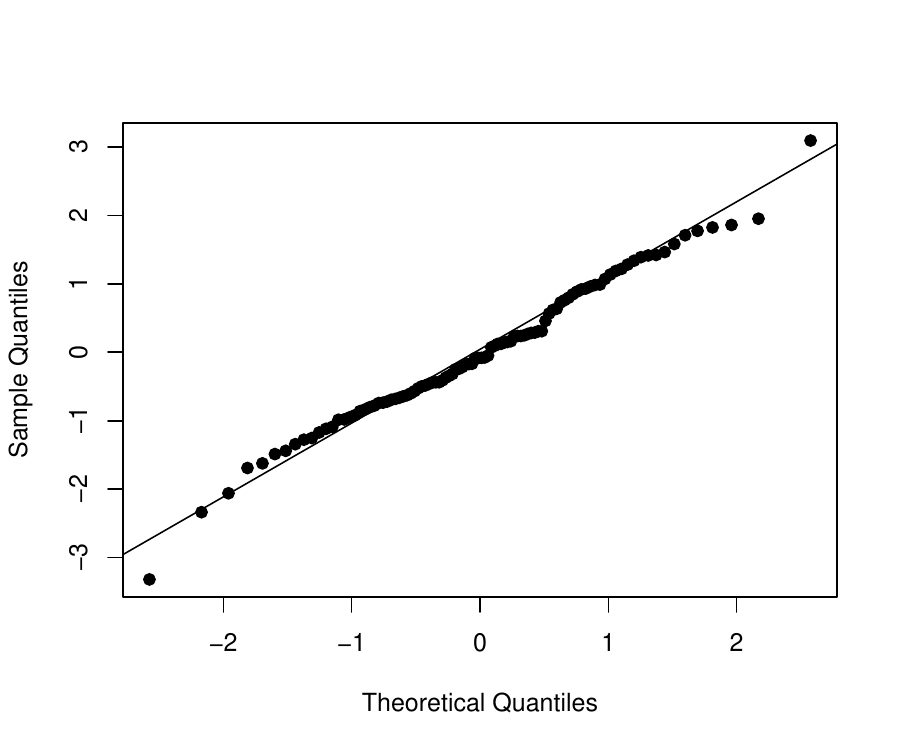}
    \end{subfigure}

    \vspace{1em}

    \begin{subfigure}{0.36\textwidth}
        \centering
        \includegraphics[width=\linewidth]{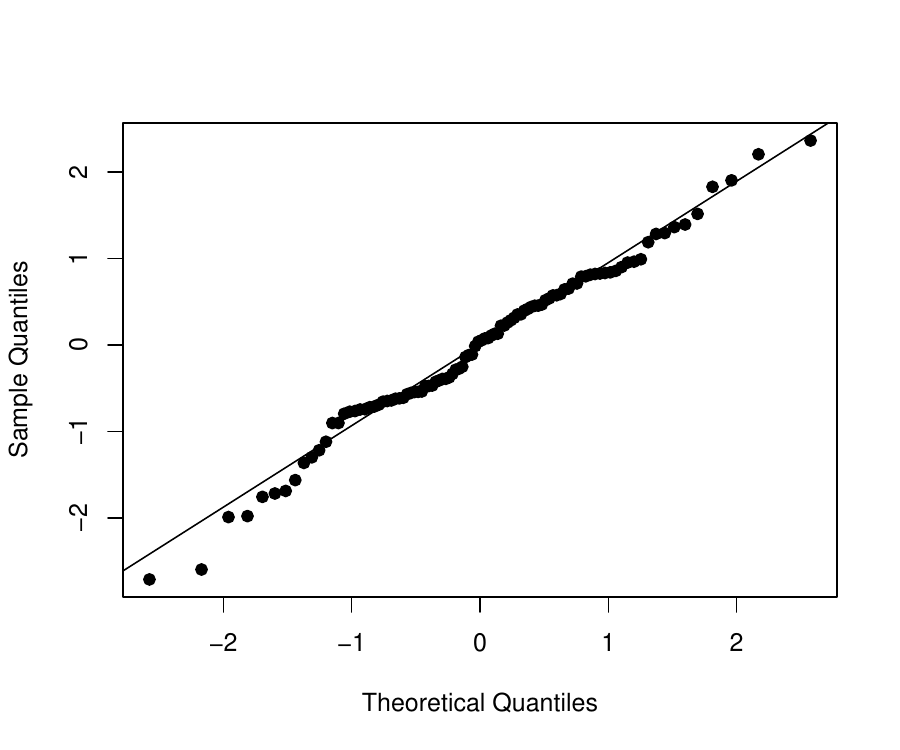}
    \end{subfigure}
    \hspace{0.02\textwidth}
    \begin{subfigure}{0.36\textwidth}
        \centering
        \includegraphics[width=\linewidth]{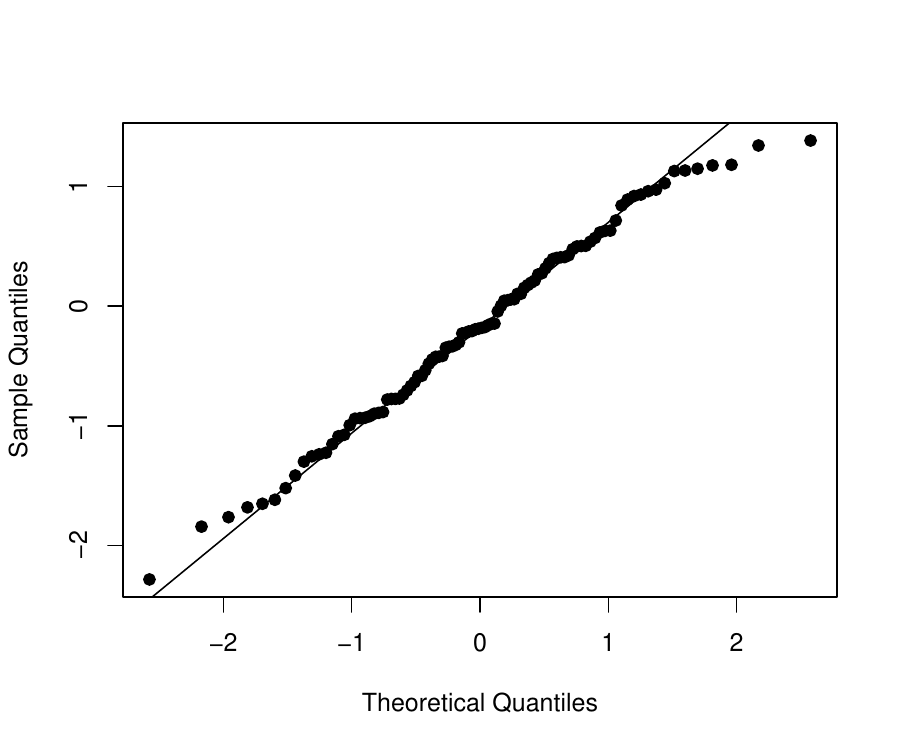}
    \end{subfigure}

    \vspace{1em}

    \begin{subfigure}{0.36\textwidth}
        \centering
        \includegraphics[width=\linewidth]{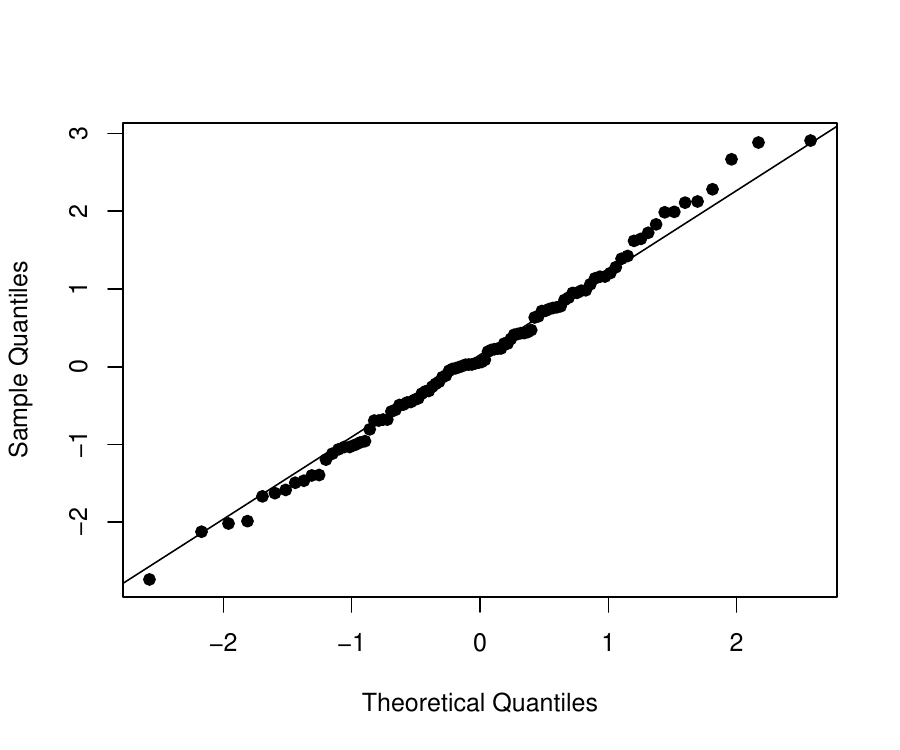}
    \end{subfigure}
    \hspace{0.02\textwidth}
    \begin{subfigure}{0.36\textwidth}
        \centering
        \includegraphics[width=\linewidth]{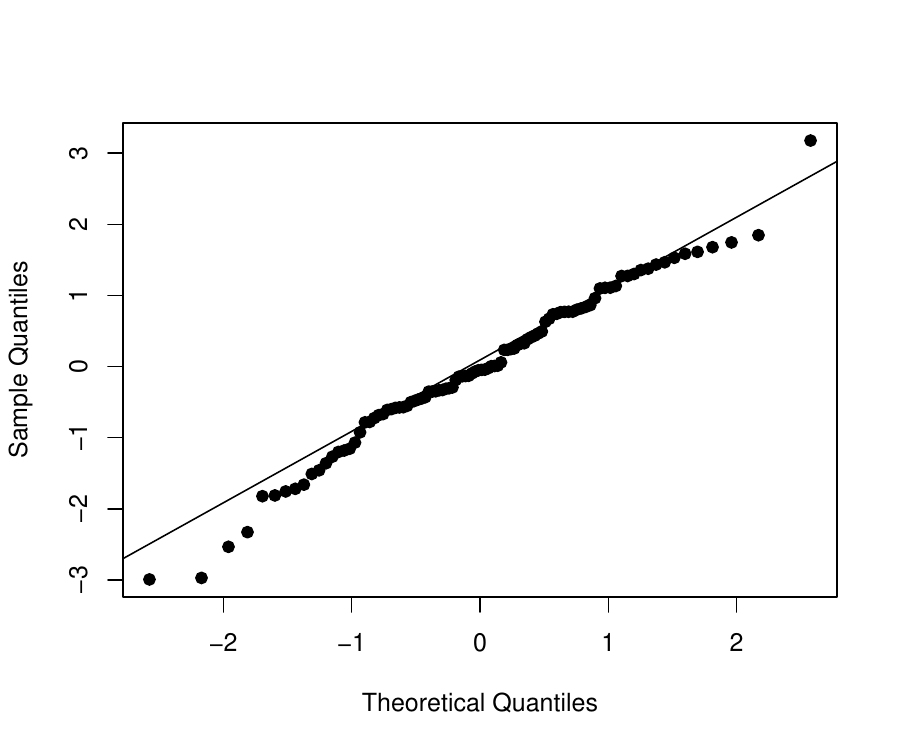}
    \end{subfigure}

    \vspace{1em}

    \begin{subfigure}{0.36\textwidth}
        \centering
        \includegraphics[width=\linewidth]{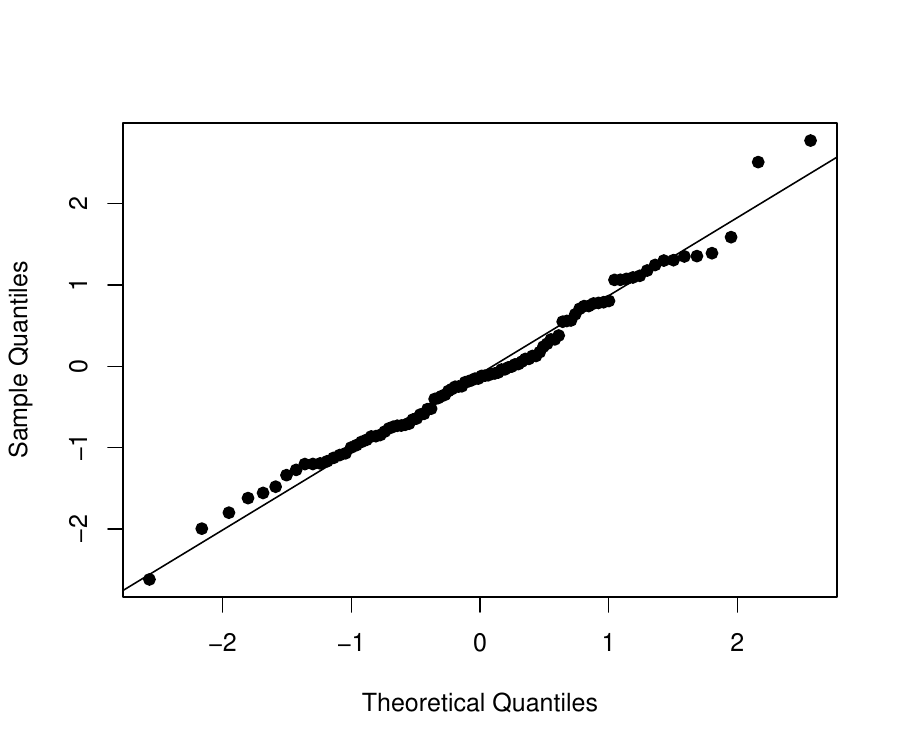}
    \end{subfigure}
    \hspace{0.02\textwidth}
    \begin{subfigure}{0.36\textwidth}
        \centering
        \includegraphics[width=\linewidth]{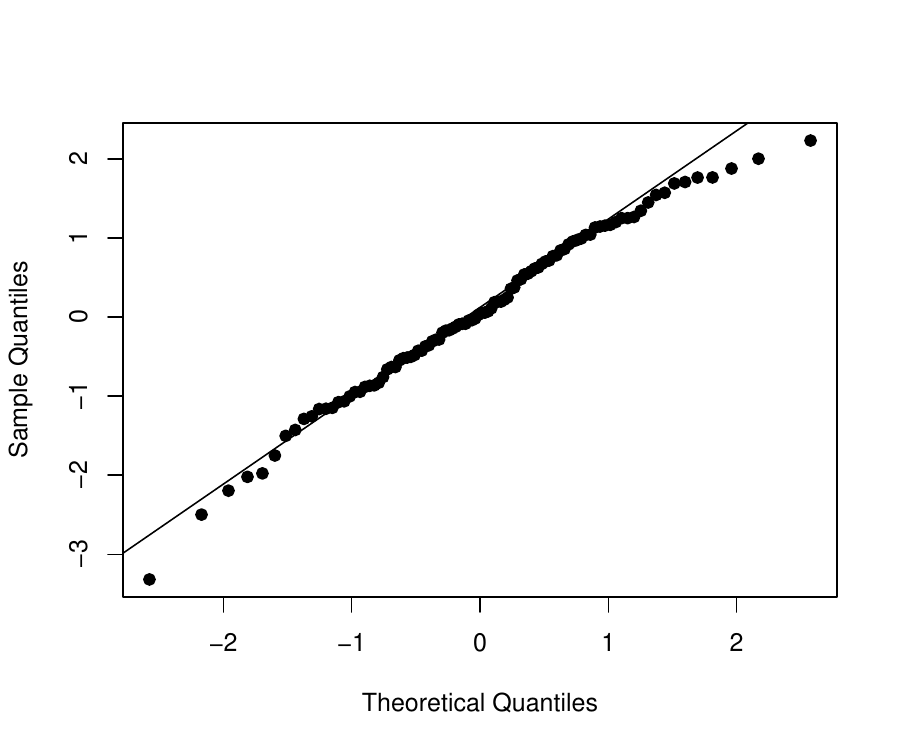}
    \end{subfigure}

    \caption{QQ plots of transformed (signed) overlaps at $d=10000$, $n=n_*$. Left column: $\theta=1$; right column: $\theta=2$. Rows correspond (top to bottom) to $\delta=0.1, 0.2, 0.6, 1$.}
    \label{fig:qq_theta_both}
\end{figure}

\medskip

Tables 1-8 show the values of $n_*=[\gamma_*d\log d]$ -- the critical number of iterations -- and $n_{\max}$ -- the total number of iterations -- for different values of $\theta,\delta,d$. The reader can immediately see for how long the algorithm needed to be run to exhibit the high dimensional effects. Hundreds of thousands of iterations were normal for smaller $\theta$.

\begin{table}[p]
    \centering
    \begin{tabular}{|c|c|c|}
        \hline
        $d$ & $n_*$ & $n_{\max}$ \\
        \hline
        500 & 16354 & 49063\\
        800 & 28146 & 84437\\
        1000 & 36357 & 109070 \\
        2000 & 80009 & 240028 \\
        5000 & 224137 & 672410 \\
        10000 & 484755 & 1454264\\
        \hline
    \end{tabular}
    \caption{$\theta=1,\delta=0.1,\gamma_*=5.26,\rho_*=0.95$}
\end{table}

\begin{table}[p]
    \centering
    \begin{tabular}{|c|c|c|}
        \hline
        $d$ & $n_*$ & $n$ \\
        \hline
        500 & 8631 & 25894\\
        800 & 14855 & 44564\\
        1000 & 19188 & 57565 \\
        2000 & 42227 & 126682 \\
        5000 & 118294 & 354883 \\
        10000 & 255843 & 767528\\
        \hline
    \end{tabular}
    \caption{$\theta=1,\delta=0.2,\gamma_*=2.78,\rho_*=0.9$}
\end{table}

\begin{table}[p]
    \centering
    \begin{tabular}{|c|c|c|}
        \hline
        $d$ & $n_*$ & $n$ \\
        \hline
        500 & 3699 & 11098\\
        800 & 6366 & 19099\\
        1000 & 8224 & 24671 \\
        2000 & 18097 & 54292 \\
        5000 & 50698 & 152093 \\
        10000 & 109647 & 328941\\
        \hline
    \end{tabular}
    \caption{$\theta=1,\delta=0.6,\gamma_*=1.19,\rho_*=0.73$}
\end{table}

\begin{table}[p]
    \centering
    \begin{tabular}{|c|c|c|}
        \hline
        $d$ & $n_*$ & $n$ \\
        \hline
        500 & 3107 & 9322\\
        800 & 5348 & 16043\\
        1000 & 6908 & 20723 \\
        2000 & 15202 & 45605 \\
        5000 & 42586 & 127758 \\
        10000 & 92103 & 276310\\
        \hline
    \end{tabular}
    \caption{$\theta=1,\delta=1,\gamma_*=1,\rho_*=0.58$}
\end{table}

\begin{table}[p]
    \centering
    \begin{tabular}{|c|c|c|}
        \hline
        $d$ & $n_*$ & $n$ \\
        \hline
        500 & 3933 & 11800\\
        800 & 6769 & 20308\\
        1000 & 8744 & 26232 \\
        2000 & 19243 & 57728 \\
        5000 & 53906 & 161719 \\
        10000 & 116587 & 349760\\
        \hline
    \end{tabular}
    \caption{$\theta=2,\delta=0.1,\gamma_*=1.27,\rho_*=0.97$}
\end{table}

\begin{table}[p]
    \centering
    \begin{tabular}{|c|c|c|}
        \hline
        $d$ & $n_*$ & $n$ \\
        \hline
        500 & 1992 & 5976\\
        800 & 3428 & 10284\\
        1000 & 4428 & 13284 \\
        2000 & 9745 & 29234 \\
        5000 & 27299 & 81896 \\
        10000 & 59041 & 177122\\
        \hline
    \end{tabular}
    \caption{$\theta=2,\delta=0.2,\gamma_*=0.64,\rho_*=0.94$}
\end{table}

\begin{table}[p]
    \centering
    \begin{tabular}{|c|c|c|}
        \hline
        $d$ & $n_*$ & $n$ \\
        \hline
        500 & 700 & 2100\\
        800 & 1204 & 3613\\
        1000 & 1556 & 4667 \\
        2000 & 3424 & 10271 \\
        5000 & 9591 & 28774 \\
        10000 & 20744 & 62232\\
        \hline
    \end{tabular}
    \caption{$\theta=2,\delta=0.2,\gamma_*=0.23,\rho_*=0.84$}
\end{table}

\begin{table}[p]
    \centering
    \begin{tabular}{|c|c|c|}
        \hline
        $d$ & $n_*$ & $n$ \\
        \hline
        500 & 444 & 1332\\
        800 & 764 & 2292\\
        1000 & 987 & 2960 \\
        2000 & 2172 & 6515 \\
        5000 & 6084 & 18251 \\
        10000 & 13158 & 39473\\
        \hline
    \end{tabular}
    \caption{$\theta=2,\delta=0.2,\gamma_*=0.14,\rho_*=0.76$}
\end{table}

\section{Variant of Oja: Effect of Spherical Gradient}\label{sec:BAGJ}

We revisit Oja's unnormalized update step from (\ref{eqn:oja_main_recursion}):
\begin{align*}
    \tilde v_k &= \hat v_{k-1} + (\delta/d) \langle X_k, \hat v_{k-1}\rangle X_k
\end{align*}

\medskip

Henceforth, for the rest of this section, we will call this \textit{ordinary} Oja; $\tilde v_k,\hat v_k$ will denote updates from ordinary Oja. Define the loss function $L(v;X)=-\langle v,X\rangle^2/2$, thus $$\nabla L(v;X)=-\langle v,X\rangle X$$ Hence, Oja's unnormalized update may be viewed as taking a gradient descent step on this loss with data $X_k$ and step size $\delta/d$:
\begin{align*}
    \tilde v_k &= \hat v_{k-1} - (\delta/d) \nabla L(\hat v_{k-1};X_k)
\end{align*}

\medskip

Recently, \citet{arous2021online} consider online stochastic SGD with loss function $L$ where the authors replace the ordinary gradient $\nabla L(v;X)$ with the spherical gradient $\nabla^\sph L(v;X):=\nabla L(v;X)-(\partial L/\partial r)(\partial v/\partial r)$ where $r=\|v\|$ is the radial part of $v$. Specializing to our loss $L(v;X)=-\langle v,X\rangle^2/2$, the spherical gradient at $v$ with $\|v\|=1$ equals
\begin{align*}
    \nabla^\sph L(v;X) &= -\langle v, X\rangle X + \langle v,X\rangle^2v
\end{align*}

\medskip

Thus, the online SGD algorithm in \citet{arous2021online} using this spherical gradient produces a variant of ordinary Oja, which we will call \textit{spherical} Oja, as follows.
\begin{align}\label{eqn:spherical_oja}
\begin{split}
    \tilde v_k^\sph &= \hat v^\sph + (\delta/d)(\langle \hat v_{k-1}^\sph, X_k\rangle X_k - \langle \hat v_{k-1}^\sph, X_k\rangle^2 \hat v_k^\sph)\\
\hat v_k^\sph &= \tilde v_k^\sph/\|\tilde v_k^\sph\|
\end{split}
\end{align}

\medskip

Given that $\hat v_k$ is constrained on the sphere $\sS^{d-1}$, one may surmise that spherical Oja is superior to ordinary Oja, as spherical Oja ``uses more structure". However, as we show in the following theorem, ordinary Oja and spherical Oja have the exact same phase transition point $\gamma_*$, the same limiting stable overlap $\rho_*$, and even the exact same non-degenerate asymptotic distribution of the overlap at criticality.

\medskip

\begin{theorem}\label{thm:BAGJ}
    Suppose, for $n\geq 1$, $X_n\stackrel{iid}{\sim}\gN(0,\Sigma)$ where $\Sigma=\theta^2v_0v_0^\top+I\in\sR^{d\times d}$. Starting at a random initialization $\hat v_0\sim Unif(\sS^{d-1})$, let $\hat v_n^\sph$ be the output from the spherical Oja algorithm after $n$ steps. Let $\gamma_*,\rho_*$ denote the phase transition location and stable overlap from the study on ordinary Oja defined in (\ref{eqn:rho_gamma_crit_defns}). Then, as $n/d\log d\to\gamma\in(0,\infty)$,
    \begin{align*}
        |\langle \hat v_n^\sph, v_0\rangle|\stackrel{p}{\to}\begin{cases}
            0, & \gamma <\gamma_*\\
            \rho_*, & \gamma > \gamma_*
        \end{cases}
    \end{align*}Further, at criticality, when $n=[\gamma_*d\log d+\eta d]$ and $d\to\infty,\eta\in \sR$, the resulting correlation is random and has the same distribution as in ordinary Oja at criticality:
    \begin{align*}
        |\langle \hat v_n^\sph,v_0\rangle|\stackrel{d}{\to}\dfrac{\rho_* |G|\exp(\eta/2\gamma_*)}{\sqrt{\rho_*^4+G^2\exp(\eta/\gamma_*)}}
    \end{align*}where $G\sim\gN(0,1)$.
\end{theorem}

The proof can be found in Section \ref{sec:proofs}. We now experimentally validate Theorem \ref{thm:BAGJ}. We take the exact same setup as in Section \ref{sec:experiments}. Data and random initialization are generated with the exact same seed as those in Section \ref{sec:experiments} with the exact same parameter values $\theta\in\{1,2\}, \delta\in\{0.1, 0.2, 0.6, 1\}$ and $d\in\{500, 800, 1000, 2000, 5000, 10000\}$. Thus the only difference between the experiments in Section \ref{sec:experiments} and those here is the difference in the algorithm.

\medskip

Figures \ref{fig:BAGJ_theta_1_delta_0.1_0.2}, \ref{fig:BAGJ_theta_1_delta_0.6_1}, \ref{fig:BAGJ_theta_2_delta_0.1_0.2} and \ref{fig:BAGJ_theta_2_delta_0.6_1} show that spherical Oja's performance is very similar to Oja's performance, with the same phase transition region and the same stable correlation $\rho_*$. Just as in Section \ref{sec:experiments}, QQ plots for the same transformation of the overlap during criticality also confirm very good agreement with $N(0,1)$ quantiles in Figure \ref{fig:qq_BAGJ_theta_both}.

\begin{figure}[p]
    \centering
    \vfill
    \begin{subfigure}{0.8\textwidth}
        \centering
        \includegraphics[width=\linewidth]{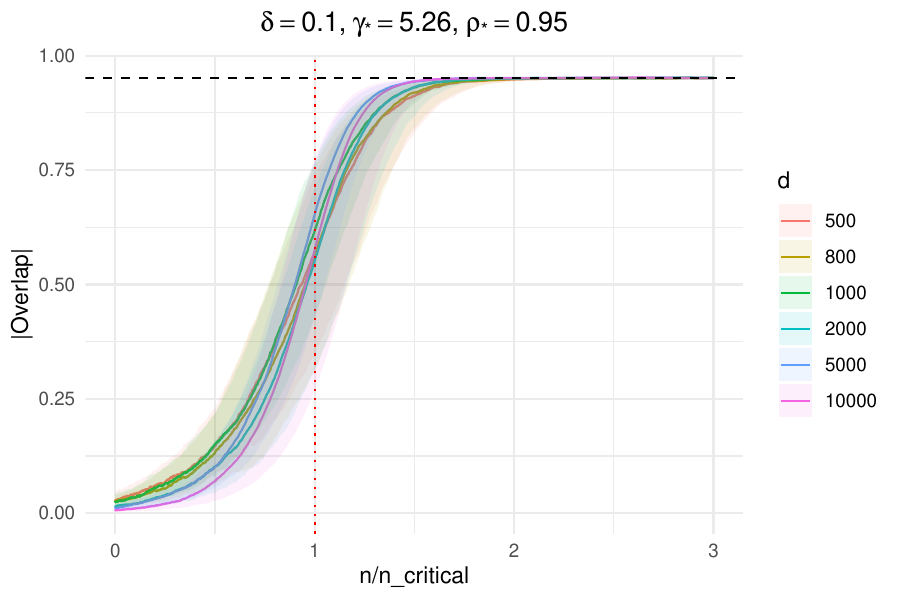}
    \end{subfigure}

    \vspace{2em}

    \begin{subfigure}{0.8\textwidth}
        \centering
        \includegraphics[width=\linewidth]{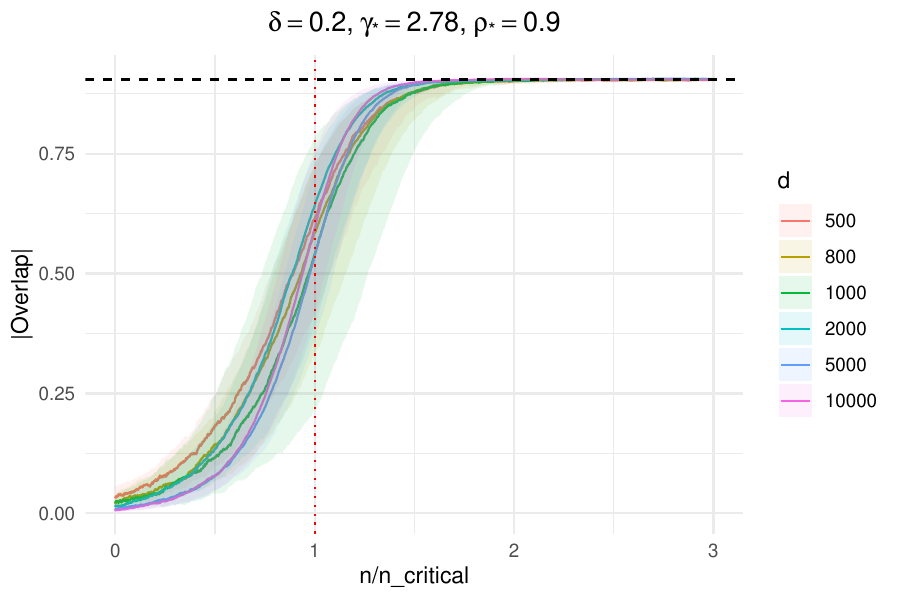}
    \end{subfigure}
    \vfill
    \caption{Performance of spherical Oja for $\theta=1$ and $\delta\in\{0.1, 0.2\}$ across different $d$. x-axis is rescaled to $n/n_*$ so that the vertical red dotted line shows the phase transition threshold at $1$. Horizontal dotted line shows $\rho_*$.}
    \label{fig:BAGJ_theta_1_delta_0.1_0.2}
\end{figure}

\begin{figure}[p]
    \centering
    \vfill
    \begin{subfigure}{0.8\textwidth}
        \centering
        \includegraphics[width=\linewidth]{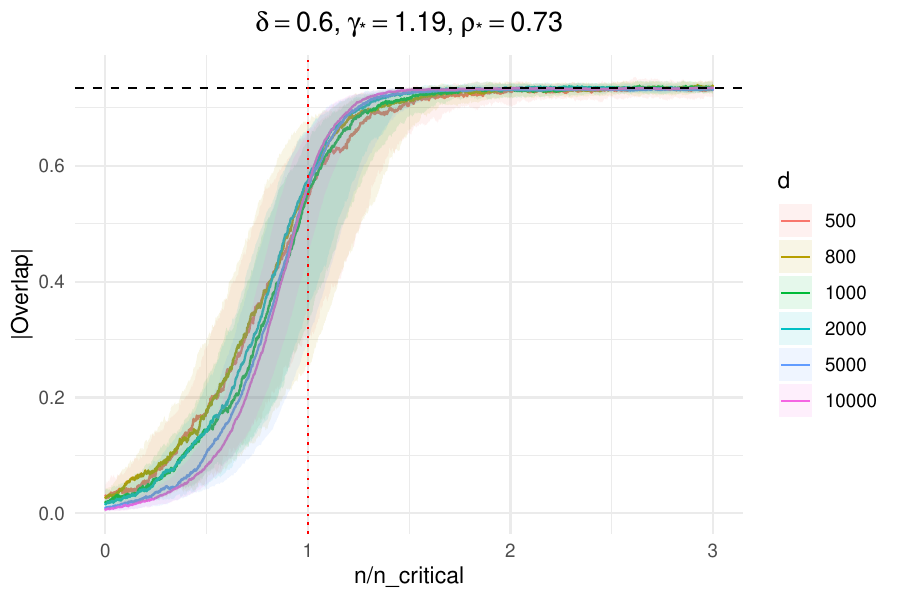}
    \end{subfigure}

    \vspace{2em}

    \begin{subfigure}{0.8\textwidth}
        \centering
        \includegraphics[width=\linewidth]{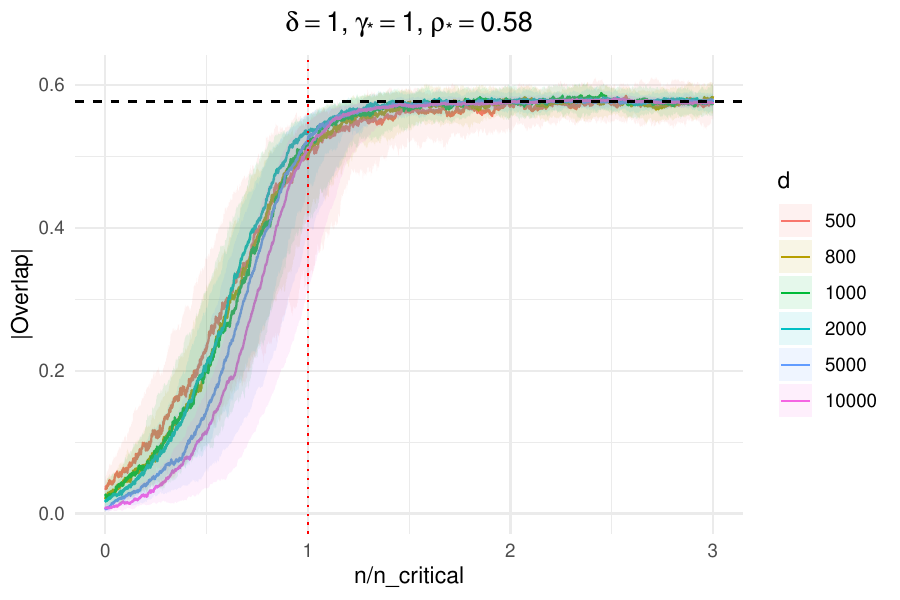}
    \end{subfigure}
    \vfill
    \caption{Performance of spherical Oja for $\theta=1$ and $\delta\in\{0.6, 1\}$ across different $d$. x-axis is rescaled to $n/n_*$ so that the vertical red dotted line shows the phase transition threshold at $1$. Horizontal dotted line shows $\rho_*$.}
    \label{fig:BAGJ_theta_1_delta_0.6_1}
\end{figure}

\begin{figure}[p]
    \centering
    \vfill
    \begin{subfigure}{0.8\textwidth}
        \centering
        \includegraphics[width=\linewidth]{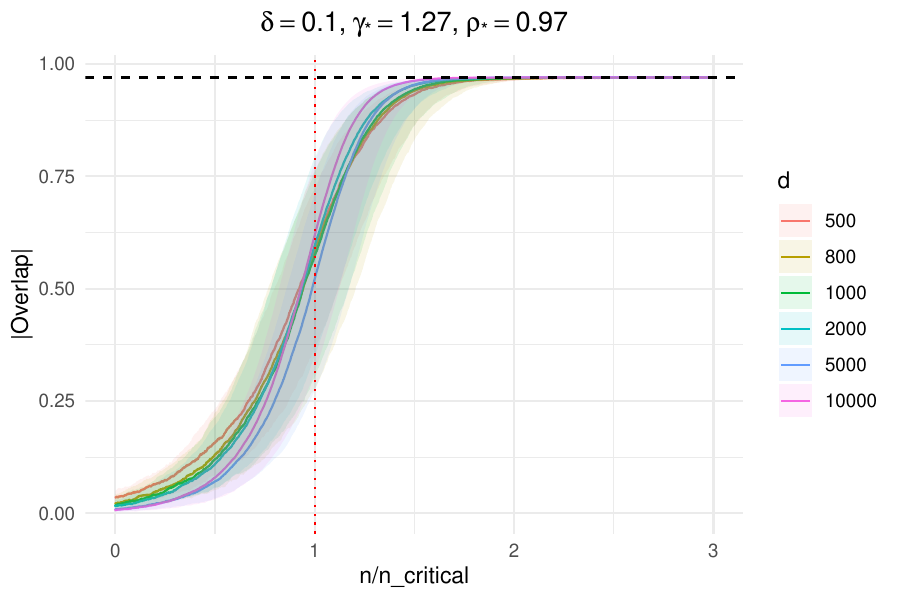}
    \end{subfigure}

    \vspace{2em}

    \begin{subfigure}{0.8\textwidth}
        \centering
        \includegraphics[width=\linewidth]{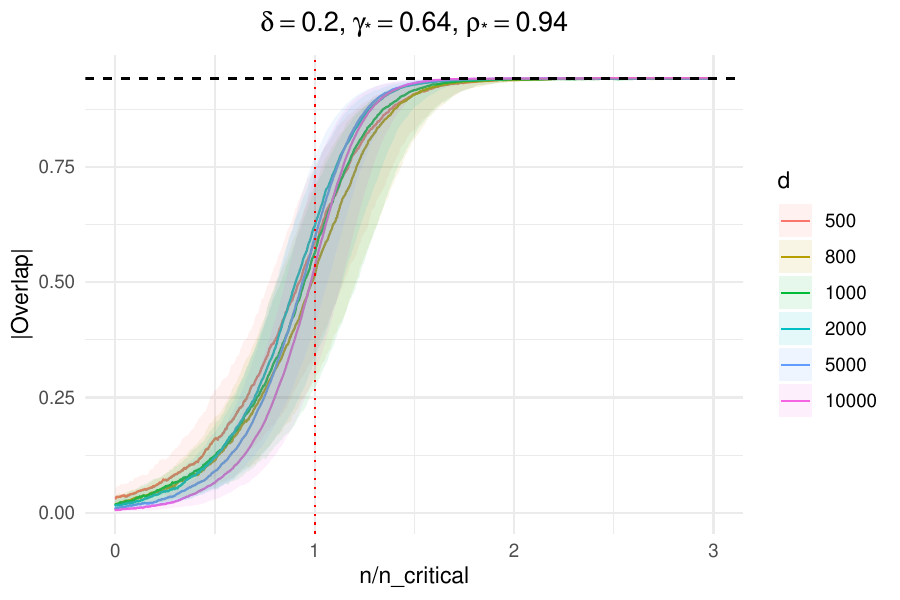}
    \end{subfigure}
    \vfill
    \caption{Performance of spherical Oja for $\theta=2$ and $\delta\in\{0.1, 0.2\}$ across different $d$. x-axis is rescaled to $n/n_*$ so that the vertical red dotted line shows the phase transition threshold at $1$. Horizontal dotted line shows $\rho_*$.}
    \label{fig:BAGJ_theta_2_delta_0.1_0.2}
\end{figure}

\begin{figure}[p]
    \centering
    \vfill
    \begin{subfigure}{0.8\textwidth}
        \centering
        \includegraphics[width=\linewidth]{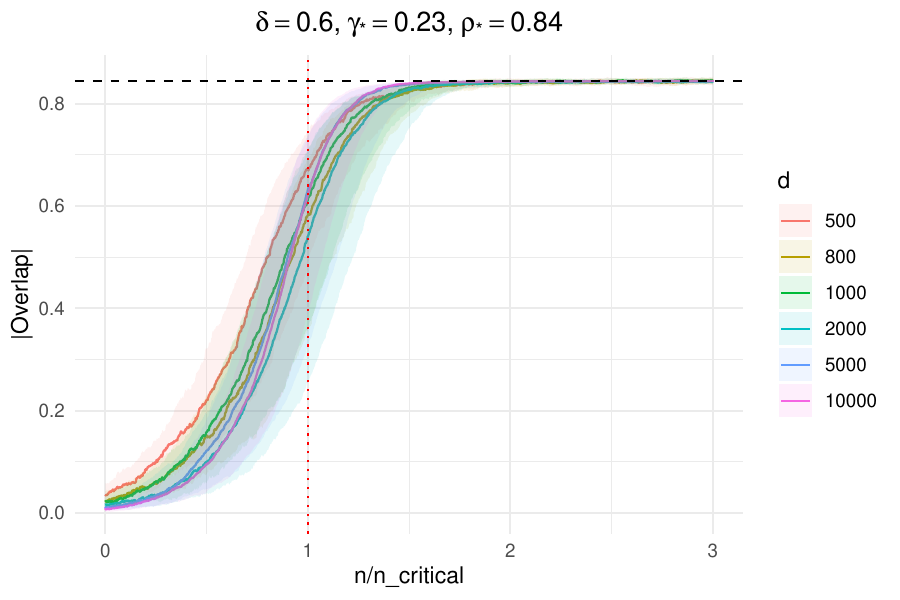}
    \end{subfigure}

    \vspace{2em}

    \begin{subfigure}{0.8\textwidth}
        \centering
        \includegraphics[width=\linewidth]{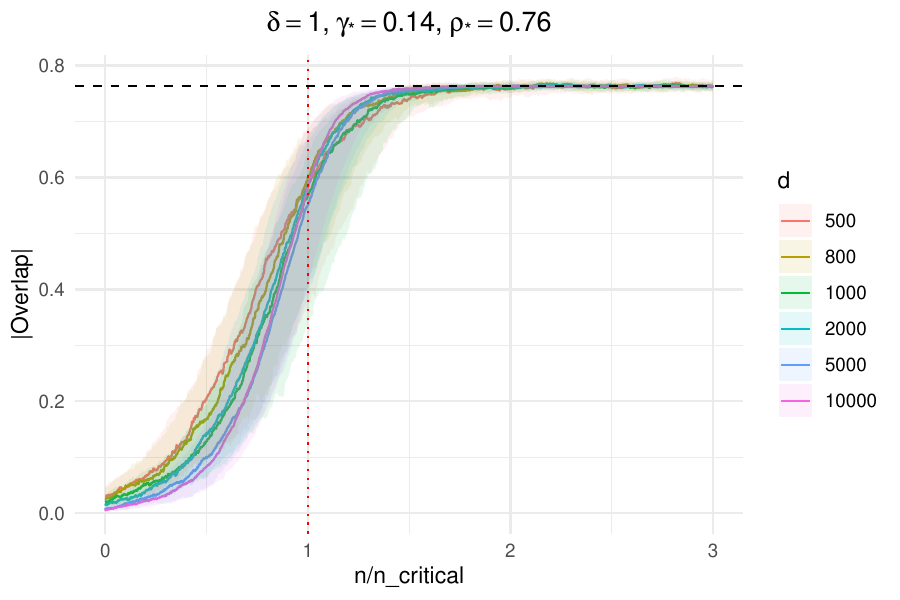}
    \end{subfigure}
    \vfill
    \caption{Performance of spherical Oja for $\theta=2$ and $\delta\in\{0.6, 1\}$ across different $d$. x-axis is rescaled to $n/n_*$ so that the vertical red dotted line shows the phase transition threshold at $1$. Horizontal dotted line shows $\rho_*$.}
    \label{fig:BAGJ_theta_2_delta_0.6_1}
\end{figure}

\begin{figure}[p]
    \centering

    \begin{subfigure}{0.36\textwidth}
        \centering
        \includegraphics[width=\linewidth]{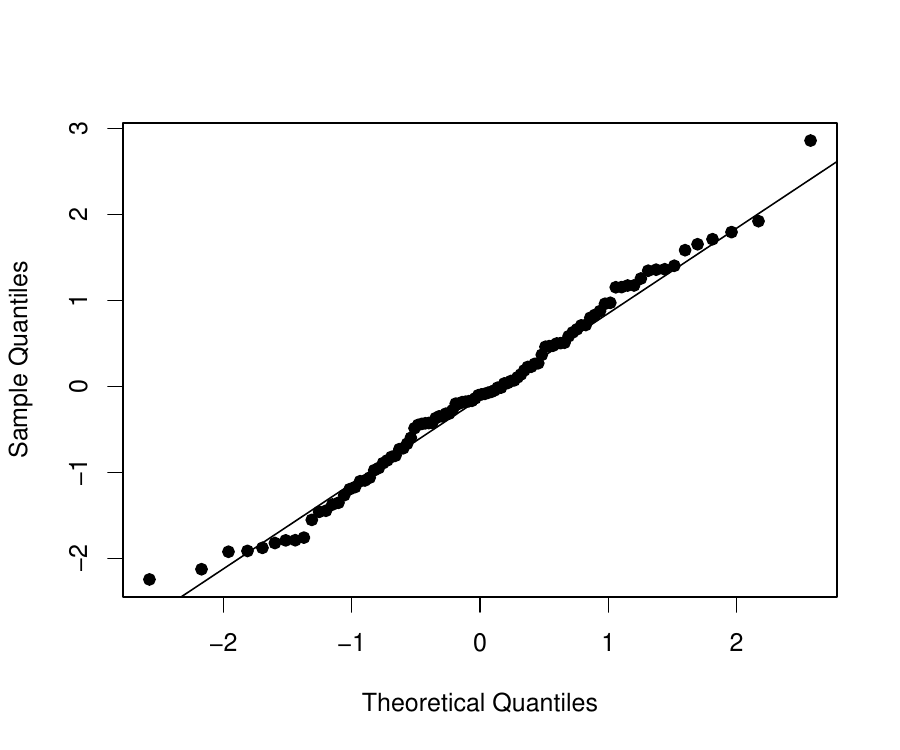}
    \end{subfigure}
    \hspace{0.02\textwidth}
    \begin{subfigure}{0.36\textwidth}
        \centering
        \includegraphics[width=\linewidth]{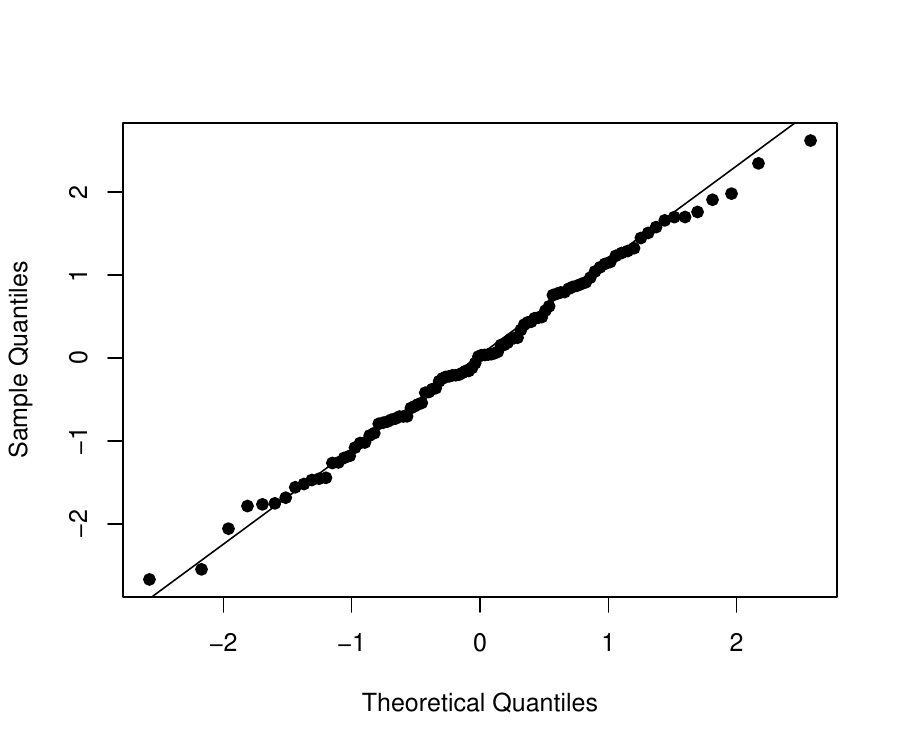}
    \end{subfigure}

    \vspace{1em}

    \begin{subfigure}{0.36\textwidth}
        \centering
        \includegraphics[width=\linewidth]{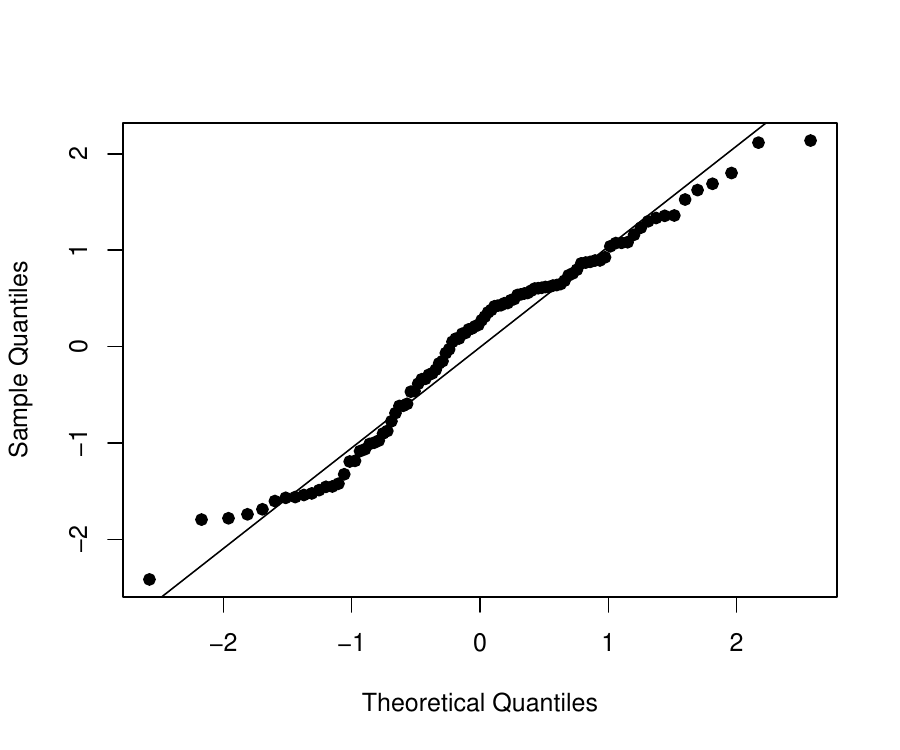}
    \end{subfigure}
    \hspace{0.02\textwidth}
    \begin{subfigure}{0.36\textwidth}
        \centering
        \includegraphics[width=\linewidth]{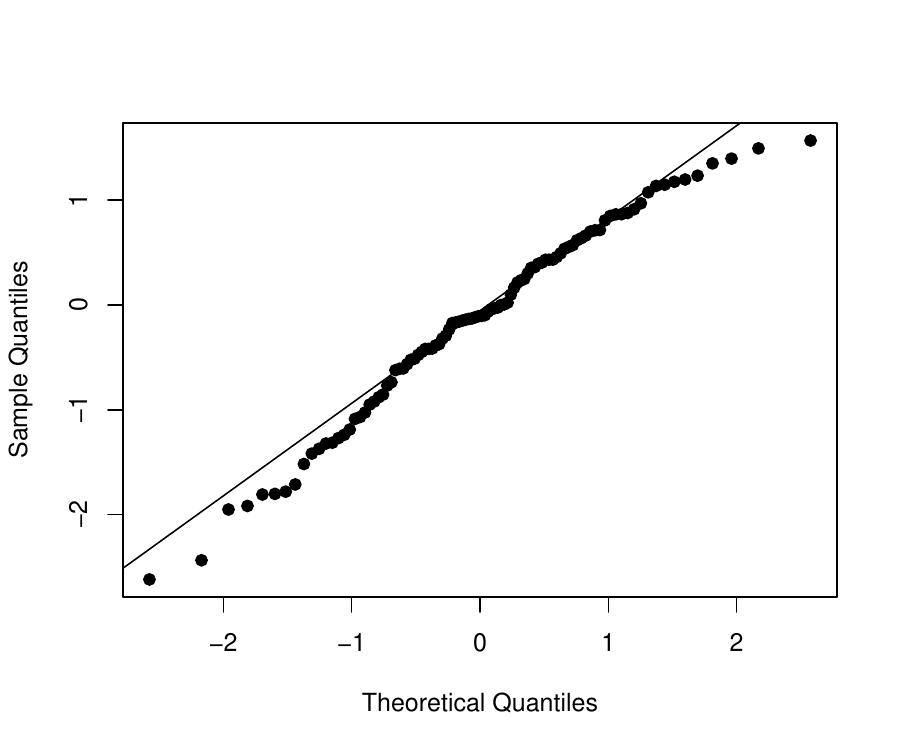}
    \end{subfigure}

    \vspace{1em}

    \begin{subfigure}{0.36\textwidth}
        \centering
        \includegraphics[width=\linewidth]{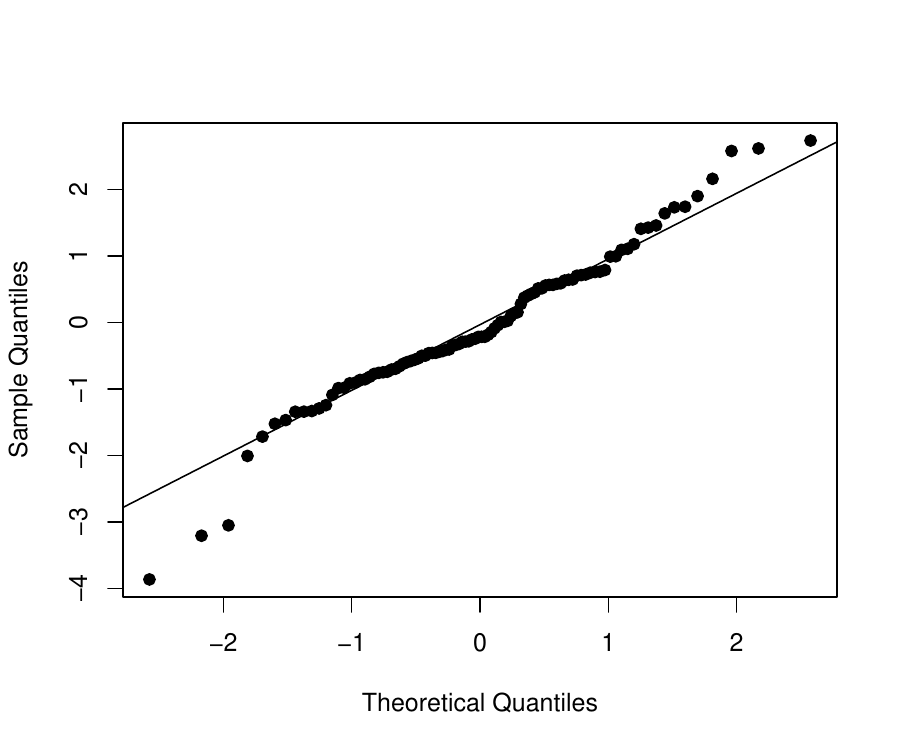}
    \end{subfigure}
    \hspace{0.02\textwidth}
    \begin{subfigure}{0.36\textwidth}
        \centering
        \includegraphics[width=\linewidth]{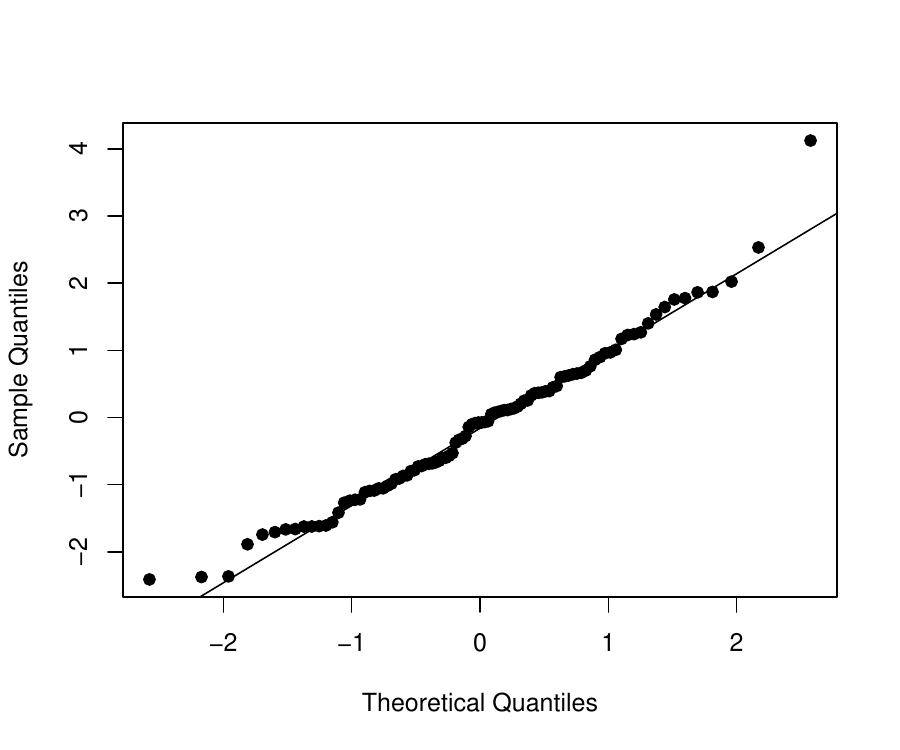}
    \end{subfigure}

    \vspace{1em}

    \begin{subfigure}{0.36\textwidth}
        \centering
        \includegraphics[width=\linewidth]{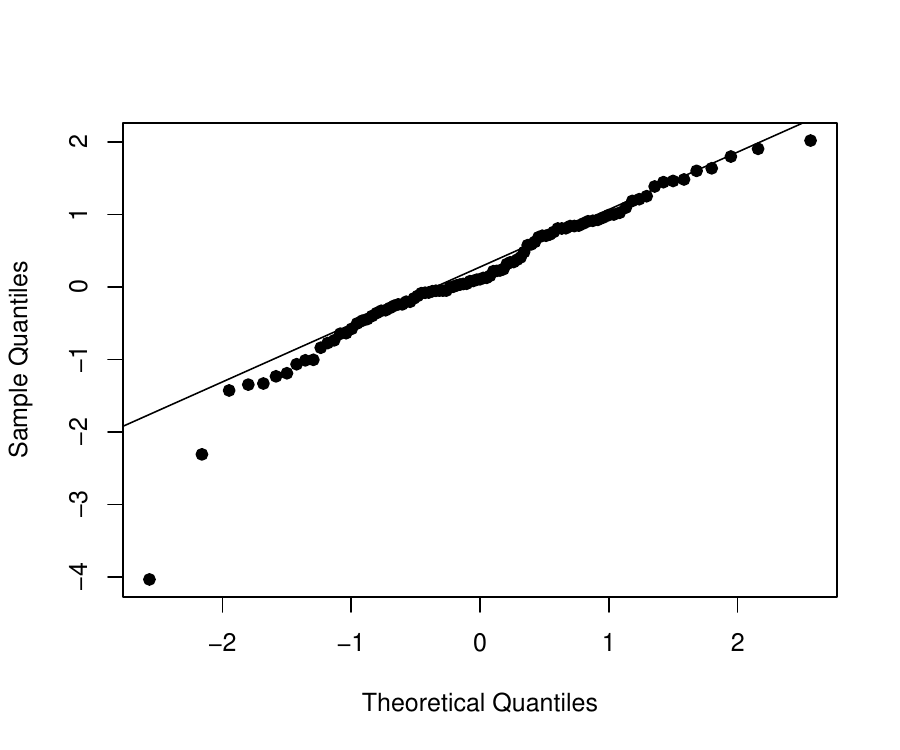}
    \end{subfigure}
    \hspace{0.02\textwidth}
    \begin{subfigure}{0.36\textwidth}
        \centering
        \includegraphics[width=\linewidth]{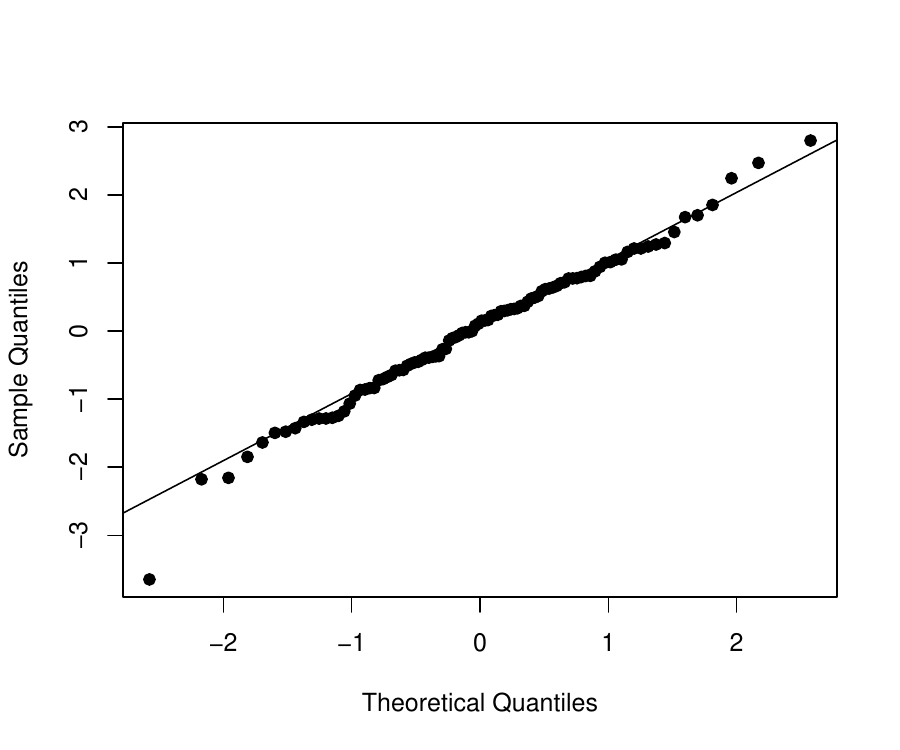}
    \end{subfigure}

    \caption{QQ plots of transformed (signed) overlaps for spherical Oja at $d=10000$, $n=n_*$. Left column: $\theta=1$; right column: $\theta=2$. Rows correspond (top to bottom) to $\delta=0.1, 0.2, 0.6, 1$.}
    \label{fig:qq_BAGJ_theta_both}
\end{figure}

\section{Proofs and Auxiliary Lemmas}\label{sec:proofs}

This section is devoted to the proofs of the main results in Section \ref{sec:main_results}, and also auxiliary lemmas and their proofs. We start with the proof of Lemma \ref{lemma:Xk_components_indep}.

\medskip

\begin{proof}[Proof of Lemma \ref{lemma:Xk_components_indep}]
    Recall that $\hat v_{k-1}=\rho_{k-1}v+\sqrt{1-\rho^2_{k-1}}e_{k-1}$ where $e_{k-1}$ is perpendicular to $v$. Then, given $\gF_{k-1}$, $(A_k,C_k)=(\langle X_k,v\rangle, \langle X_k,e_{k-1}\rangle)$ is a mean zero Gaussian vector with
    \begin{align*}
        Var(A_k|\gF_{k-1})&=v_0^\top \Sigma v_0 = \theta^2 + 1,\\
        Var(C_k|\gF_{k-1})&=e_{k-1}^\top \Sigma e_{k-1} = 1,\\
        Cov(A_k,C_k|\gF_{k-1})&=v^\top\Sigma e_{k-1}=0
    \end{align*}
    Since the distribution of $(A_k,C_k)$ does not depend on $\gF_{k-1}$, this proves that $(A_k,C_k)$ is independent of $\gF_{k-1}$, and unconditionally,
    \begin{align*}
        \begin{pmatrix}
            A_k\\
            C_k
        \end{pmatrix}\sim \gN\begin{pmatrix}
            \begin{pmatrix}
                0 \\
                0
            \end{pmatrix}, \begin{pmatrix}
                \theta^2+1 & 0 \\
                0 & 1
            \end{pmatrix}
        \end{pmatrix}
    \end{align*}

Next, conditional on $\gF_{k-1}$, $D_k=P_{k-1}^\perp X_k\sim \gN(0, P_{k-1}^\perp)$ since $$Var(P_{k-1}^\perp X_k|\gF_{k-1})=P_{k-1}^\perp\Sigma P_{k-1}^\perp=P_{k-1}^\perp$$Consequently, we may write $D_k|\gF_{k-1}\stackrel{d}{=} P^{\perp}_{k-1}Z_k$ for $Z_k\sim\gN(0,I_d)$ independent of $\gF_{k-1}$.

It is a standard result from multivariate analysis (see \citet{mardia2024multivariate} for example) that if $C$ is an orthogonal projection matrix and $Z\sim \gN(0,I)$, then $Z^\top CZ\sim\chi^2_{\text{rank(C)}}$. Taking $C=P_{k-1}^\perp$ in our case, $\text{rank(C)}=d-2$. Thus, conditional on $\gF_{k-1}$, $\|D_k\|^2\sim \chi^2_{d-2}$. Again, since the distribution of $\|D_k\|^2$ does not depend on $\gF_{k-1}$, it follows that $\|D_k\|^2$ is independent of $\gF_{k-1}$ and is unconditionally $\chi^2_{d-2}$.

Finally, we need to show that $D_k$ is independent of $(A_k,C_k)$, given $\gF_{k-1}$. Towards this, note that given $\gF_{k-1}$, $(A_kv, C_ke_{k-1},D_k)$ is multivariate Gaussian with mean $0$. We now execute the following conditional covariance computations.
\begin{align*}
    Cov(D_k,A_k|\gF_{k-1}) &= \E[\langle X_k,v_0\rangle P_{k-1}^\perp X_k|\gF_{k-1}]\\
    &=\E[P_{k-1}^\perp X_k X_k^\top v_0|\gF_{k-1}]\\
    &= P_{k-1}^\perp \Sigma v_0\\
    &= 0
\end{align*}
\begin{align*}
    Cov(D_k,C_k|\gF_{k-1}) &= \E[\langle X_k,e_{k-1}\rangle P_{k-1}^\perp X_k|\gF_{k-1}]\\
    &=\E[P_{k-1}^\perp X_k X_k^\top e_{k-1}|\gF_{k-1}]\\
    &= P_{k-1}^\perp \Sigma e_{k-1}\\
    &= 0
\end{align*}

Since $A_k,C_k,D_k$ are all conditionally zero mean Gaussian given $\gF_{k-1}$ and all covariances are zero, this implies the independence of $D_k$ from $(A_kv,C_ke_{k-1})$. This in particular implies that $\|D_k\|^2$ is independent of $(A_k,C_k)$ given $\gF_{k-1}$, and hence is conditionally independent of $B_k=\rho_{k-1}A_k+\sqrt{1-\rho^2_{k-1}}C_k$. But all three variables, $A_k, C_k,\|D_k\|^2$ are all independent of $\gF_{k-1}$, which finally implies that $A_k,C_k,\|D_k\|^2$ are unconditionally independent. This completes the proof.
\end{proof}

\medskip

Now, we come to the proof of the important Lemma \ref{lemma:main_recursion}, which underlies the major results in this work.

\medskip

\begin{proof}[Proof of Lemma \ref{lemma:main_recursion}]
    Let $f(x)=(1+x)^{-1/2}$ for $x\geq 0$. By Taylor expansion,
    \begin{align*}
        f(x) &= f(0) + f'(0)x + f''(\xi)\dfrac{x^2}{2}
    \end{align*}where $0\leq \xi\leq x$. Now $f(0)=1, f'(0)=-1/2$ and $f''(\xi)=3(1+\xi)^{-5/2}/4$. Thus, $0\leq f''(\xi)\leq 1$. Consequently, starting with recursion \ref{eqn:recursion_ABC}, we get
    \begin{align}\label{eqn:rho_expand}
    \begin{split}
        \rho_k &= \left(\rho_{k-1}+\dfrac{\delta}{d}A_kB_k\right)f\left(\dfrac{\delta B_k^2}{d}\left(2 + \dfrac{\delta\|X_k\|^2}{d}\right)\right)\\
        &=\left(\rho_{k-1} + \dfrac{\delta}{d}A_kB_k\right)\left(1 - \dfrac{\delta B_k^2}{2d}\left(2 + \dfrac{\delta\|X_k\|^2}{d}\right)\right) + \dfrac{r_k^{(1)}}{d^2}
    \end{split}
    \end{align}where
    \begin{align}\label{eqn:rk}
        r_k^{(1)} &= \left(\rho_{k-1}+\dfrac{\delta}{d}A_kB_k\right)\times f''(\xi)\times \dfrac{1}{2}\left(\dfrac{\delta B_k^2}{2}\left(2 + \dfrac{\delta\|X_k\|^2}{d}\right)\right)^2
    \end{align}and so 
    \begin{align*}
        |r_k^{(1)}| &\leq C\left(1 + \dfrac{\delta}{d}|A_k||B_k|\right)\times B_k^4\times\left(4 + \dfrac{\delta^2\|X_k\|^4}{d^2}\right)
    \end{align*}Using Lemma \ref{lemma:Xk_components_indep} and the distributions of $A_k,B_k,X_k$, we get that there is a universal constant $C$ such that $\E|r_k^{(1)}|\leq C$ for all $k$.

We now simplify (\ref{eqn:recursion_ABC}) even further by isolating lower order terms. Let $r_k^{(2)}=\delta^2A_kB_k^3(2 + \|X_k\|^2/d)/2$. Also, we write $B_k=\rho_{k-1}A_k + \sqrt{1-\rho^2_{k-1}}C_k$, and plugging this in,
\begin{align*}
    \rho_k &= \rho_{k-1} + \dfrac{\delta}{d}\rho_{k-1}A_k^2 + \dfrac{\delta}{d}\sqrt{1-\rho_{k-1}^2}A_kC_k -\dfrac{\delta B_k^2}{2d}\left(2 + \dfrac{\delta\|X_k\|^2}{d}\right)\rho_{k-1} - r_k^{(2)} + r_k^{(1)}
\end{align*}
Let $r_k^{(3)}=\delta^2 B_k^2(A_k^2+C_k^2)\rho_{k-1}/2$. We thus further simplify
\begin{align*}
    \rho_k &= \rho_{k-1} + \dfrac{\delta}{d}\rho_{k-1}A_k^2 + \dfrac{\delta}{d}\sqrt{1-\rho_{k-1}^2}A_kC_k -\dfrac{\delta B_k^2}{2d}\left(2 + \dfrac{\|D_k\|^2}{d}\right)\rho_{k-1} -r_k^{(3)}- r_k^{(2)} + r_k^{(1)}
\end{align*}

Set $R_k=r_k^{(1)}-r_k^{(2)}-r_k^{(3)}$. Since $\E|r_k^{(i)}|^j\leq C_j$ for each $i\in\{1,2,3\}$ and each $k$, it follows that $\E|R_k|^j\leq C_j$ for a constant $C_j$, for any $k$.

Finally, we define $M_k$ by
\begin{align}\label{eqn:martingale_def}
\begin{split}
    M_k &= \delta\rho_{k-1}(A_k^2-(\theta^2+1))+\delta\sqrt{1-\rho^2_{k-1}}A_kC_k-\delta\rho_{k-1}(B_k^2-(\theta^2\rho_{k-1}^2+1))\\
    &-\delta^2\rho_{k-1}(B_k^2\|D_k\|^2/d-(\theta^2\rho_{k-1}^2+1))/2
\end{split}
\end{align}Clearly $M_k$ is $\gF_k$-measurable and one can check that $\E[M_k|\gF_{k-1}]=0$, using Lemma \ref{lemma:Xk_components_indep}. Thus $\{M_k\}_{k\geq 1}$ a martingale difference sequence. Further, from the expression (\ref{eqn:martingale_def}), it is clear that $M_k$ is a polynomial in the random variables $A_k,B_k,C_k,D_k^2/d$ which all have finite moments of all orders. This implies that each moment $\E|M_k|^j\leq C_j$ for all $k$. Some algebra yields the final expression, and the proof is complete. The unfolded result is a simple consequence.
\end{proof}



\medskip

Before proceeding further, it would be valuable to understand how large $\hat \rho_k$ can be. The following lemma provides the relevant bound.

\medskip

\begin{lemma}\label{lemma:uniform_bounds_rho}
    There is a universal constant $C>0$ such that $\E|\rho_k|^3\leq C\exp(3\alpha k/d)/d^{3/2}$.
\end{lemma}

\medskip

\begin{proof}[Proof of Lemma \ref{lemma:uniform_bounds_rho}] We start with the recursion given in Lemma \ref{lemma:main_recursion}. Squaring both sides, taking conditional expectation given $\gF_{k-1}$ and bounding terms of order lower than $1/d$, we get
\begin{align*}
    \E[\rho_k^2|\gF_{k-1}] &\leq \rho_{k-1}^2 + 2\rho_{k-1}^2(\alpha-\beta\rho^2_{k-1})/d + C/d^2\\
    &\leq \rho^2_{k-1} + 2\alpha\rho^2_{k-1}/d + C/d^2
\end{align*}for a universal constant $C>0$. Thus, setting $\mu_{2,k}=\E[\rho^2_k]$, we get
\begin{align*}
    \mu_{2,k}&\leq \mu_{2,k-1}(1 + 2\alpha/d) + C/d^2
\end{align*}Iterating this and using $\mu_{2,0}=1/d$, we get
\begin{align*}
    \mu_{2,k} &\leq \mu_{2,0}(1+2\alpha/d)^k + \dfrac{C}{d^2}\sum_{l=0}^{k-1} (1+2\alpha/d)^l \\
    &\leq \dfrac{1}{d}\exp\left(2\alpha k/d\right) + \dfrac{C}{2\alpha d}\exp(2\alpha k/d)\\
    &\leq \dfrac{C}{d}\exp(2\alpha k/d)
\end{align*}Next, raising to the fourth power the recursion in Lemma \ref{lemma:main_recursion}, 
\begin{align*}
    \E[\rho^4_k|\gF_{k-1}] &\leq \rho^4_{k-1}(1 + 4\alpha/d) + C\rho^2_{k-1}/d^2 + C/d^3
\end{align*}Writing $\mu_{4,k}=\E[\rho^4_k]$, we then get
\begin{align*}
    \mu_{4,k} &\leq \mu_{4,k-1}(1+4\alpha/d) + C\mu_{2,k-1}/d^2 + C/d^3
\end{align*}We iterate this all the way to $k=0$. Note that $\mu_{4,0}=\E[\langle Z,v\rangle^4/\|Z\|^4]=\E[Beta^2(1,d-1)]$, where $Z\sim \gN(0, I_d)$. Recall that for $a,b>0$, $\E[Beta^2(a,b)]=a(a+1)/(a+b)(a+b+1)$. Plugging in $a=1,b=d-1$, we get $\mu_{4,0}=2/d(d+1)\leq 2/d^2$. Using this and also using the previous estimate $\mu_{2,k}\leq C\exp(2\alpha k/d)/d$, we get
\begin{align*}
    \mu_{4,k} &\leq \mu_{4,0}\left(1 + 4\alpha/d\right)^k + \dfrac{C}{d^2}\sum_{l=0}^{k-1}\left(1 + 4\alpha/d\right)^l\mu_{2,k-1-l} + \dfrac{C}{d^3}\sum_{l=0}^{k-1}(1+4\alpha/d)^l \\
    &\leq \dfrac{2}{d^2}\exp(4\alpha k/d) + \dfrac{C}{d^3}\sum_{l=0}^{k-1}\exp(4\alpha l/d)\exp(2\alpha(k-1-l)/d) + \dfrac{C}{d^2}\exp(4\alpha k/d)\\
    &\leq \dfrac{C}{d^2}\exp(4\alpha k/d) + \dfrac{C\exp(2\alpha(k-1)/d)}{d^2}\sum_{l=0}^{k-1}\exp(2\alpha l/d)\\
    &\leq \dfrac{C}{d^2}\exp(4\alpha k/d) + \dfrac{C\exp(2\alpha(k-1)/d)}{d^3}\dfrac{\exp(2\alpha k/d)}{\exp(2\alpha/d)-1}\\
    &\leq \dfrac{C}{d^2}\exp(4\alpha k/d) + \dfrac{C}{d^2}\exp(4\alpha k/d)
\end{align*}where in the last step we used the inequality $\exp(2\alpha/d)-1\geq 2\alpha/d$ in the denominator. A standard application of Lyapounov's inequality then concludes the proof: $\E|\rho_k|^3 \leq (\E[\rho_k^4])^{3/4}\leq C\exp(3\alpha k/d)/d^{3/2}$.
\end{proof}

Lemma \ref{lemma:uniform_bounds_rho} immediately implies Theorem \ref{thm:alpha0}.

\begin{proof}[Proof of Theorem \ref{thm:alpha0}]
    Consider the bound $\E|\rho_n|^3\leq C\exp(3\alpha n/d)/d^{3/2}$ from Lemma \ref{lemma:uniform_bounds_rho}. When $\alpha<0$, the upper bound $\exp(3\alpha n/d)/d^{3/2}\leq 1/d^{3/2}\to 0$ as $d\to\infty$. This holds irrespective of $n$. The proof is complete.
\end{proof}

An additional lemma will be helpful in simplifying the asymptotic distribution of the martingale sum $\sum_{k=1}^n c_d^{n-k}M_k/d$ in \ref{eqn:recursion_rho_unfolded_final}. Looking at the expression for $M_k$ in \eqref{eqn:martingale_def}, we observe that the dominant term is $\tilde M_k:=\delta A_kC_k$, since the rest of the terms are multiplied by $\rho_k$ which is typically small. The following lemma makes this intuition precise.

\begin{lemma}\label{lemma:martingale_simplified}
    Recall the martingale difference $M_k$ defined in (\ref{eqn:martingale_def}). Let $\tilde M_k=\delta A_kC_k$. Then, if $\limsup_{d\to\infty} n/d\log d\leq \gamma_*$,
    \begin{align*}
        \dfrac{1}{d}\sum_{k=1}^n c_d^{n-k}(M_k-\tilde M_k)\stackrel{p}{\to}0
    \end{align*}
\end{lemma}

\begin{proof}[Proof of Lemma \ref{lemma:martingale_simplified}]
    Write $S_k=M_k-\tilde M_k$. Note that $S_k$ are martingale differences adapted to the same filtration $\gF_k$. Thus, $\E[S_k|\gF_{k-1}]=0$ for each $k$.

    Next, observe that the exact form of $S_k$ is
    \begin{align*}
     \begin{split}
    S_k &= \delta\left(\sqrt{1-\rho^2_{k-1}}-1\right)A_kC_k+\delta\rho_{k-1}(A_k^2-(\theta^2+1))-\delta\rho_{k-1}(B_k^2-(\theta^2\rho_{k-1}^2+1))\\
    &-\delta^2\rho_{k-1}(B_k^2D_k^2-(\theta^2\rho_{k-1}^2+1))/2\\
\end{split}
    \end{align*}
    Note that $|1-\sqrt{1-\rho^2_{k-1}}|\leq \rho^2_{k-1}$, so that $\delta|1-\sqrt{1-\rho^2_{k-1}}||A_kC_k|\leq \delta\rho_{k-1}^2|A_kC_k|$. The other terms all contribute a factor of $\rho_{k-1}$. Thus, in totality, we have
    \begin{align}\label{ineq:Sk_bound}
        |S_k| &\leq \rho_{k-1}\times(\text{terms with bounded moments})
    \end{align}
   In fact, for any $j$, it is easy to see that we have $\E|S_k|^j\leq C_j|\rho_{k-1}|^j$ for a universal constant $C_j$. This implies, for each $k$, $\E(S_k^2|\gF_{k-1})\leq C\rho^2_{k-1}$. From the proof of Lemma \ref{lemma:uniform_bounds_rho}, it follows that $\E(\rho_k^2)\leq C\exp(2\alpha k/d)/d$. Using these facts and also that $c_d\leq \exp(\alpha/d)$,
    \begin{align*}
        \E\left(\dfrac{1}{d}\sum_{k=1}^nc_d^{n-k}S_k\right)^2 &\leq \dfrac{1}{d^3}\sum_{k=1}^n \exp(2\alpha (n-k)/d)\times \exp(2\alpha k/d)\\
        &\leq \dfrac{Cn\exp(2\alpha n/d)}{d^3}
    \end{align*}Since $\limsup n/d\log d\leq \gamma_*$, given $\epsilon\in(0,1)$ there exists a positive integer $d_0$ such that for all $d\geq d_0$, we have $n\leq (1+\epsilon)\gamma_*d\log d$. Thus, for all $d\geq d_0$, the upper bound can be further bounded as
    \begin{align*}
        \dfrac{Cn\exp(2\alpha n/d)}{d^3}\leq \dfrac{C(d\log d)\exp((1+\epsilon)\log d)}{d^3}\leq \dfrac{C\log d}{d^{1-\epsilon}}\to 0
    \end{align*}This implies that $\sum_{k=1}^n c_d^{n-k}S_k/d\stackrel{p}{\to}0$.
\end{proof}

\medskip

Lemma \ref{lemma:martingale_simplified} allows us to ``replace" the martingale difference $M_k$ by the simplified martingale difference $\tilde M_k$, which are actually iid and independent of $\rho_0$! Hence, for all future purposes, we can rewrite recursion \ref{eqn:recursion_final} as
\begin{align}\label{eqn:recursion_final_simplified}
    \rho_k &= (1+\alpha/d)\rho_{k-1} - \beta\rho^3_{n-1}/d + \tilde M_k/d + S_k/d + R_k/d^2
\end{align}and recursion \ref{eqn:recursion_rho_unfolded_final} as
\begin{align}\label{eqn:recursion_unfolded_final_simplified}
    \rho_k &= c_d^k\rho_0 - \dfrac{\beta}{d}\sum_{k=1}^n c_d^{n-k}\rho^3_{k-1} + \dfrac{1}{d}\sum_{k=1}^n c_d^{n-k}\tilde M_k + \dfrac{1}{d}\sum_{k=1}^n c_d^{n-k}S_k + \dfrac{1}{d^2}\sum_{k=1}^n c_d^{n-k}R_k
\end{align}

\medskip

\begin{proof}[Proof of Theorem \ref{thm:subcritical}] 
We start with the form of the recursion in \ref{eqn:recursion_unfolded_final_simplified}. Recall that $\sqrt{d}\rho_0\stackrel{w}{\to}\gN(0,1)$ as $d\to\infty$, and hence $\sqrt{d}\rho_0$ is stochastically bounded. Now
\begin{align*}
    \dfrac{c_d^n}{\sqrt{d}}\leq \dfrac{\exp(\alpha n/d)}{d^{1/2}}
\end{align*}and since $\lim_{d\to\infty} n/d\log(d)=\gamma<\gamma_*=1/2\alpha$, we get
\begin{align*}
    \limsup_{d\to\infty}\dfrac{c_d^n}{d^{1/2}}\leq \exp\left(\alpha \limsup_{d\to\infty}\log(d)\left(\gamma-\gamma_*\right)\right)=0
\end{align*}This implies that the first term $c_d^n\rho_0=(c_d^n/\sqrt{d})\times \sqrt{d}\rho_0\stackrel{p}{\to}0$.

Using Lemma \ref{lemma:uniform_bounds_rho}, we can bound
\begin{align*}
    \E\left|\dfrac{1}{d}\sum_{k=1}^n c_d^{n-k}\rho_{k-1}^3\right|&\leq \dfrac{1}{d}\sum_{k=1}^n c_d^{n-k}\E|\rho_{k-1}|^3\\
    &\leq \dfrac{C}{d^{5/2}}\sum_{k=1}^n\exp(\alpha(n-k)/d)\times \exp(3\alpha k/d)\\
    &=\dfrac{C\exp(\alpha (n+2)/d)}{d^{5/2}}\sum_{k=0}^{n-1}\exp(2\alpha k/d)\\
    &\leq \dfrac{C\exp(\alpha(n+2)/d)}{d^{5/2}}\times \dfrac{\exp(2\alpha n/d)}{\exp(2\alpha/d)-1}\\
    &\leq \dfrac{C\exp(\alpha n/d)}{d^{3/2}}\times \exp(2\alpha n/d)\\
    &\leq \dfrac{\exp(3\alpha n/d)}{d^{3/2}}
\end{align*}Since $\exp(\alpha n/d)/d^{1/2}\to0$, it follows that $\exp(3\alpha n/d)/d^{3/2}\to 0$ as well, and hence
\begin{align*}
\dfrac{1}{d}\sum_{k=1}^n c_d^{n-k}\rho_{k-1}^3\stackrel{p}{\to}0
\end{align*}
Now we study the third term. Recall that $\tilde M_k$ are iid terms of the form $\tilde M_k=\delta A_kC_k$ where $A_k\sim\gN(0,\theta^2+1)$ and $C_k\sim \gN(0,1)$ independent of each other. Thus, 
\begin{align*}
    \E\left(\dfrac{1}{d}\sum_{k=1}^n c_d^{n-k}\tilde M_k\right)^2 &= \dfrac{\delta^2(\theta^2+1)}{d^2}\sum_{k=0}^{n-1} c_d^{2k}\\
    &\leq \dfrac{C\exp(2\alpha n/d)}{d}\to 0
\end{align*}following the logic in the preceding paragraph. This implies that
\begin{align*}
    \dfrac{1}{d}\sum_{k=1}^n c_d^{n-k}\tilde M_k\stackrel{p}{\to}0
\end{align*}

The fourth term $\sum_{k=1}^n c_d^{n-k}S_k/d$ has been already shown to be asymptotically small in Lemma \ref{lemma:martingale_simplified}. Finally, we tackle the last term. Since $\E|R_k|\leq C$ for all $k$, we get
\begin{align*}
    \E\left|\dfrac{1}{d^2}\sum_{k=1}^n c_d^{n-k}R_k\right| \leq \dfrac{C}{d^2}\sum_{k=0}^{n-1}c_d^k\leq \dfrac{C\exp(\alpha n/d)}{d}\to 0
\end{align*}since $\exp(\alpha n/d)/d^{1/2}\to 0$ and hence 
\begin{align*}
\dfrac{1}{d^2}\sum_{k=1}^n c_d^{n-k}R_k\stackrel{p}{\to}0    
\end{align*}


Thus, each term on the right side of recursion \ref{eqn:recursion_rho_unfolded_final} converges to $0$ in probability, thereby completing the proof.
\end{proof}

\medskip

While Theorem \ref{thm:subcritical} establishes non-recovery below $\gamma_*$, it turns out that the conclusion extends to $n=[\gamma_*d\log(d)-t_dd]$ for a slowly increasing sequence $t_d$ that still goes to $\infty$ but at a rate slower than $\log(d)$. Moreover, we can precisely identify the precise behavior of $\rho_n$ for this value of $n$. Correctly scaled, it has Gaussian fluctuations. This is documented in Lemma \ref{lemma:exact_behavior_rho_subcritical}.

\medskip

\begin{lemma}\label{lemma:exact_behavior_rho_subcritical}
    Let $n=[\gamma_*d\log(d)-t_dd]$ where $t_d\to\infty,t_d/\log(d)\to0$ as $d\to\infty$. Then,
    \begin{align*}
        \exp(\alpha t_d)\rho_n\stackrel{d}{\to}\gN(0, 1+\gamma_*\delta^2(\theta^2+1))
    \end{align*}
\end{lemma}

\medskip

\begin{proof}[Proof of Lemma \ref{lemma:exact_behavior_rho_subcritical}]
    We start with recursion \ref{eqn:recursion_unfolded_final_simplified}, and multiply both sides by $\exp(\alpha t_d)$. Then, following the proof of Theorem \ref{thm:subcritical},
    \begin{align*}
        \E\left|\dfrac{\exp(\alpha t_d)}{d}\sum_{k=1}^n c_d^{n-k}\rho^3_{k-1}\right|&\leq \dfrac{C\exp(\alpha t_d+3\alpha n/d)}{d^{3/2}}\\
        &\leq \dfrac{C\exp(\alpha t_d + 3\log(d)/2-3\alpha t_d)}{d^{3/2}}\\
        &=C\exp(-2\alpha t_d)\to 0
    \end{align*}This implies that 
    \begin{align*}
        \dfrac{\exp(\alpha t_d)}{d}\sum_{k=1}^n c_d^{n-k}\rho^3_{k-1}\stackrel{p}{\to}0
    \end{align*}
    Hence, even after scaling up by $\exp(\alpha t_d)$, the ``cubic" term involving $\rho_k^3$ (for $k\leq n-1$) is asymptotically negligible. Similarly, again from the proof of Theorem \ref{thm:subcritical},
    \begin{align*}
        \E\left|\dfrac{\exp(\alpha t_d)}{d^2}\sum_{k=1}^n c_d^{n-k}R_k\right| &\leq \dfrac{C\exp(\alpha t_d+\log(d)/2-\alpha t_d)}{d}= \dfrac{C}{\sqrt{d}}\to 0
    \end{align*}and thus $$\exp(\alpha t_d)\sum_{k=1}^n c_d^{n-k}R_k/d^2\stackrel{p}{\to}0$$
    Finally, we show that
    \begin{align*}
        \dfrac{\exp(\alpha t_d)}{d}\sum_{k=1}^n c_d^{n-k}S_k\stackrel{p}{\to}0
    \end{align*}To see this, we follow the proof of Lemma \ref{lemma:martingale_simplified}, and obtain
    \begin{align*}
        \E\left(\dfrac{\exp(\alpha t_d)}{d}\sum_{k=1}^n c_d^{n-k}S_k\right)^2 &= \dfrac{\exp(2\alpha t_d)}{d^2}\sum_{k=1}^n c_d^{2(n-k)}\E(S_k^2)\\
        &\leq \dfrac{C\exp(2\alpha t_d)}{d^3}\sum_{k=1}^n \exp(2\alpha(n-k)/d)\times \exp(2\alpha k/d)\\
        &\leq \dfrac{Cn\exp(2\alpha n/d)}{d^3}
    \end{align*}Since $n\leq \gamma_*d\log d$, it follows that the right side is further bounded by $C\log(d)/d\to 0$ and hence we get the desired result.

    Therefore, we are left with the ``initial" term $\exp(\alpha t_d)c_d^n\rho_0$ and the ``independent sum" term $\exp(\alpha t_d)\sum_{k=1}^n c_d^{n-k}M_k/d$. We will show that they converge to independent $N(0,1)$ and $N(0, \gamma_*(\theta^2+1)\delta^2)$. Together with the smallness of the two other terms, an application of Slutsky's lemma then implies the desired result.

    Note that
    \begin{align*}
    n\log(1+\alpha/d)+\alpha t_d-\log(d)/{2} &= n(\alpha/d + O(1/d^2)) +\alpha t_d -\log(d)/2)\\
    &=O(n/d^2)=O(\log d/d)\to 0
    \end{align*}This implies that 
    \begin{align*}
        \dfrac{\exp(\alpha t_d)\times c_d^n}{d^{1/2}}= \exp\left(n\log(1+\alpha/d)+\alpha t_d-\log(d)/{2}\right)\to 1
    \end{align*}Further, we already know that $d^{1/2}\rho_0\stackrel{d}{\to}\gN(0,1)$. Thus, we get $\exp(\alpha t_d)c_d^n\rho_0\stackrel{d}{\to}\gN(0,1)$.

    Finally, we apply the Lyapounov central limit theorem to show that
\begin{align*}
    \dfrac{\exp(\alpha t_d)}{d}\sum_{k=1}^n c_d^{n-k}\tilde M_k\stackrel{d}{\to}\gN(0,\gamma_*(\theta^2+1)\delta^2)
\end{align*} Observe that the terms $\tilde M_k=\delta A_kC_k$ are iid with mean $0$, variance $\delta^2(\theta^2+1)$ and bounded moments of all orders. Consequently, the variance of this sum yields
\begin{align*}
    \E\left(\dfrac{\exp(\alpha t_d)}{d}\sum_{k=1}^nc_d^{n-k}\tilde M_k\right)^2 &= \dfrac{\exp(2\alpha t_d)\delta^2(\theta^2+1)}{d^2}\sum_{k=0}^n c_d^{2k}\\
    &=\dfrac{\delta^2(\theta^2+1)\exp(2\alpha t_d)(\exp(2\alpha n/d)-1)}{2\alpha}
\end{align*}Now, using that $\gamma_*d\log d-t_d d-1\leq n\leq \gamma_*d\log d-t_d d$, it follows that
\begin{align*}
    \log d - 2\alpha t_d - 2\alpha/d \leq 2\alpha n/d\leq \log d - 2\alpha t_d
\end{align*}which finally implies that
\begin{align*}
    \dfrac{\exp(2\alpha t_d + 2\alpha n/d) }{d}\to 1
\end{align*}and hence, 
\begin{align*}
    \E\left(\dfrac{\exp(\alpha t_d)}{d}\sum_{k=1}^nc_d^{n-k}\tilde M_k\right)^2 \to \gamma_*\delta^2(\theta^2+1)
\end{align*} The last step is to verify the Lyapounov condition is satisfied, that is, we need to show that
\begin{align*}
    \dfrac{\exp(3\alpha t_d)}{d^3}\sum_{k=1}^nc_d^{3(n-k)}\E|\tilde M_k|^3\to 0
\end{align*}But each $\E|\tilde M_k|^3\leq C$, so that
\begin{align*}
    \dfrac{\exp(3\alpha t_d)}{d^3}\sum_{k=1}^nc_d^{3(n-k)}\E|\tilde M_k|^3 &\leq \dfrac{C\exp(3\alpha t_d)\times \exp(3\alpha n/d)}{d^2}\\
    &\leq \dfrac{C\exp(3\alpha \gamma_*\log d)}{d^2}\\
    &=\dfrac{Cd^{3/2}}{d^2}\to 0
\end{align*}and the proof is complete.\end{proof}

\medskip

Thus, at $n=n_-=[\gamma_*d\log d - t_d d]$, $\rho_n$ still vanishes but not too fast; the rate at which it decays to $0$ as $d\to\infty$ is precisely $\exp(-\alpha t_d)$ which can be quite slow if $t_d\to\infty$ slowly enough. The fact that the algorithm now has reached an overlap level that is significantly better than the $O(1/\sqrt{d})$ order afforded by random initialization, will turn out to be quite important for what follows.

\medskip

Before proceeding, we recall a discrete Gronwall's lemma (stated in equation (5.1) in \citet{arous2021online}). We also provide the proof for convenience.

\medskip

\begin{lemma}\label{lemma:gronwall}
    Let $\{m_k\}_k$ be a sequence of reals such that for $a,b.\geq 0$, $$m_k\leq a + b\sum_{l=0}^{k-1}m_l$$ for all $k$. Then, $m_k\leq a\exp(bk)$ for all $k$.
\end{lemma}

\begin{proof}[Proof of Lemma \ref{lemma:gronwall}]
    We proceed by strong induction. The statement is clearly true when $k=0$. Suppose it is true for all $l\leq k-1$, so that $m_l\leq a\exp(bl)$ for all $l\leq k-1$. Then, 
    \begin{align*}
        m_k &\leq a + b\sum_{l=0}^{k-1}m_l\\
        &\leq a\left(1 + b\sum_{l=0}^{k-1}\exp(bl)\right)\\
        &= a\left(1 + \dfrac{b(\exp(bk)-1)}{\exp(b)-1}\right)
    \end{align*}We next use that $b\leq \exp(b)-1$, and thus the induction step is proven.
\end{proof}

\medskip

We now show, in the next couple lemmas, that once $n\geq n_-$, the trajectory of $\rho_n$ follows a deterministic differential equation. This equation is the famous logistic ODE \citep{robinson2004introduction} that comes up in a classroom discussion of pitchfork bifurcation. We note that this ODE also comes up in the work of \citet{li2017diffusion}. This enables us to precisely calculate the limiting overlap on linear timescale $n=[\gamma_*d\log d+\eta d]$.

\medskip

\begin{lemma}\label{lemma:ODE_solution}
Let $\rho(t)$ be the (smooth) curve following the first order differential equation
    \begin{align}\label{eqn:ODE}
        \dfrac{d\rho(t)}{dt} &= \rho(t)(\alpha - \beta\rho^2(t))
    \end{align}starting at $\rho(0)$ satisfying $\rho^2(0)\in(0, \alpha/\beta)$. The solution $\rho(t)$ is then given by
    \begin{align*}
        \rho(t) &= \dfrac{\sqrt{\alpha}\rho(0)\exp(\alpha t)}{\sqrt{\alpha +\beta\rho^2(0)(\exp(2\alpha t)-1)}}
    \end{align*}
\end{lemma}

\begin{proof}[Proof of Lemma \ref{lemma:ODE_solution}]
    Write $y\equiv \rho(t)$ for brevity. We separate the variables and integrate:
    \begin{align*}
        t &= \int_{\rho(0)}^{\rho(t)}\dfrac{dy}{y(\alpha - \beta y^2)}\\
        &= \dfrac{1}{2\alpha}\left[2\log\left|\dfrac{\rho(t)}{\rho(0)}\right|-\log\left|\dfrac{\alpha -\beta \rho^2(t)}{\alpha -\beta\rho^2(0)}\right|\right]
    \end{align*}A standard phase analysis (see for example \citet{robinson2004introduction}) of the function on the right side of the ODE, $g(x)=x(\alpha -\beta x^2)$, establishes that since $0<\rho^2(0)<\alpha/\beta$, we have $0\leq \rho^2(t)\leq \alpha/\beta$ and also $\sign(\rho(t))=\sign(\rho(0))$ for all $t$. This enables us to conclude, after some algebra, that the solution to the ODE equals
    \begin{align*}
        \rho(t) &= \dfrac{\sqrt{\alpha}\rho(0)\exp(\alpha t)}{\sqrt{\alpha +\beta\rho^2(0)(\exp(2\alpha t)-1)}}
    \end{align*}
\end{proof}

\medskip

\begin{lemma}\label{lemma:ode_tracking}
    Let $\rho(t)$ be the (smooth) curve following the first order differential equation
    \begin{align}\label{eqn:ODE}
        \dfrac{d\rho(t)}{dt} &= \rho(t)(\alpha - \beta\rho^2(t))
    \end{align}starting at $\rho(0)=\rho_{n_-}$. Then, for any $k\leq 2t_d d$,
    \begin{align*}
        \E|\rho_{n_-+k}-\rho(k/d)|&\leq C\left(\sqrt{\dfrac{t_d}{d}}+\dfrac{t_d}{d}\right)\exp(Ct_d)
    \end{align*}
\end{lemma}

\begin{proof}[Proof of Lemma \ref{lemma:ode_tracking}]
    Let $g(x)=x(\alpha - \beta x^2)$ so that we may rewrite the differential equation \ref{eqn:ODE} as
    \begin{align*}
        \rho'(t) &= g(\rho(t))
    \end{align*}
    By Lemma \ref{lemma:main_recursion},
    \begin{align}\label{eqn:ODE_discrete}
        \rho_{n_-+k} &= \rho_{n_-+k-1}+\dfrac{g(\rho_{n_-+k-1})}{d}+\dfrac{M_k}{d}+\dfrac{R_k}{d^2}
    \end{align}
    By Taylor expansion of the function $\rho(k/d)$ around $(k-1)/d$, we get
    \begin{align}\label{eqn:ODE_taylor}
    \begin{split}
    \rho(k/d) &= \rho((k-1)/d) + \dfrac{\rho'((k-1)/d)}{d} + \dfrac{\rho''(\xi_{k,d})}{2d^2} \\
    &= \rho((k-1)/d) + \dfrac{g(\rho((k-1)/d))}{d} + \dfrac{\rho''(\xi_{k,d})}{2d^2}
    \end{split}
    \end{align}where $\xi_{k,d}\in[(k-1)/d,k/d]$. Let $e_k=\rho_{n_-+k}-\rho(k/d)$ for each $k$. Then, subtracting \ref{eqn:ODE_taylor} from \ref{eqn:ODE_discrete}, we get
    \begin{align*}
        e_k &= e_{k-1} + \dfrac{1}{d}[g(\rho_{n_-+k-1})-g(\rho((k-1)/d))] + \dfrac{M_k}{d} + \dfrac{1}{2d^2}\left(2R_k+\rho''(\xi_{k,d})\right)
    \end{align*}Unfolding this recursion and using $e_0=0$ (since, recall, $\rho(0)=\rho_{n_-}$) we get
    \begin{align*}
        e_k &= \dfrac{1}{d}\sum_{l=0}^{k-1}[g(\rho_{n_-+l})-g(\rho(l/d))] + \dfrac{1}{d}\sum_{l=1}^k M_l + \dfrac{1}{2d^2}\sum_{l=1}^k (2R_l+\rho''(\xi_{l,d}))
    \end{align*} We will now take absolute value and then expectation on both sides to get
    \begin{align*}
        \E|e_k| &\leq \dfrac{1}{d}\sum_{l=0}^{k-1}\E|g(\rho_{n_-+l})-g(\rho(l/d))| + \dfrac{1}{d}\E\left|\sum_{l=1}^k M_l\right| + \dfrac{1}{2d^2}\sum_{l=1}^k \E|2R_l+\rho''(\xi_{l,d})|
    \end{align*}
    
    Note that $g'(x)=\alpha - 3\beta x^2$, and thus for $|x|\leq 1$, $g'$ is uniformly bounded by a constant, say $C$, implying that $g$ is Lipschitz: $|g(x)-g(y)|\leq C|x-y|$ whenever $-1\leq x,y\leq 1$. Thus, for all $0\leq l\leq k$,
    \begin{align*}
        |g(\rho_{n_-+l})-g(\rho(l/d))| \leq C|\rho(n_-+l)-\rho(l/d)| = C|e_l|
    \end{align*}
    Next, since $\sum_{l=1}^k M_l/d$ is a martingale with each term $\E[M_l^2]\leq C$,
    \begin{align*}
        \E\left|\sum_{l=1}^k M_l\right| \leq \left(\E\left(\sum_{l=1}^k M_l\right)^2\right)^{1/2} \leq C\sqrt{k}
    \end{align*}Finally, $\E|R_l|\leq C$ for all $l$. Further, $$\rho''(x)=g'(\rho(x))\rho'(x)=\rho(x)(\alpha-3\beta \rho^2(x))(\alpha - \beta \rho^2(x))$$so that $|\rho''(x)|\leq C$ for all $x$. Putting all this together,
    \begin{align*}
        \E|e_k| &\leq \dfrac{C}{d}\sum_{l=0}^{k-1}\E|e_l| + \dfrac{C\sqrt{k}}{d} + \dfrac{Ck}{d^2}
    \end{align*}Then, Lemma \ref{lemma:gronwall} implies that
    \begin{align*}
        \E|e_k|\leq C\left(\dfrac{\sqrt{k}}{d}+\dfrac{k}{d^2}\right)\exp(Ck/d)
    \end{align*}Since the upper bound is increasing in $k$, it is maximized when $k=2t_d d$, and we get the desired conclusion.
\end{proof}


\medskip

We are now ready to prove the distribution of the limiting overlap at criticality.

\medskip

\begin{proof}[Proof of Theorem \ref{thm:critical}]
        Let $t_d\to\infty$ sufficiently slowly such that the right side in Lemma \ref{lemma:ode_tracking} goes to $0$ as $d\to\infty$. One choice may be $t_d=\log\log d$. Then, Lemma \ref{lemma:ode_tracking} implies that for any $k\leq 2t_d d$, $\rho_{n_-+k}-\rho(k/d)\stackrel{p}{\to}0$. In particular, taking $k=(t_d+\theta)d$, and correspondingly $n=[\gamma_*d\log d + \theta d]$, we get
    \begin{align*}
        \rho_n -\rho(t_d+\theta)\stackrel{p}{\to}0
    \end{align*}But from the explicit form of $\rho(t)$ from Lemma \ref{lemma:ODE_solution},
    \begin{align*}
        \rho(t_d+\eta) &= \dfrac{\sqrt{\alpha}\rho_{n_-}\exp(\alpha (t_d+\eta))}{\sqrt{\alpha + \beta\rho^2_{n_-}(\exp(2(\alpha t_d+ \eta)-1)}}
    \end{align*}By Lemma \ref{lemma:exact_behavior_rho_subcritical}, $\exp(\alpha t_d)\rho_{n_-}\stackrel{d}{\to}\gN(0, v^2)$ with $v^2=1+\gamma_*\delta^2(\theta^2+1)$. By the continuous mapping theorem, this implies that
    \begin{align*}
        \rho(t_d+\eta) \stackrel{d}{\to}\dfrac{\sqrt{\alpha}v\exp(\alpha\eta)G}{\sqrt{\alpha + \beta v^2 \exp(2\alpha\eta) G^2}}
    \end{align*}where $G\sim\gN(0,1)$. An application of Slutsky's lemma then transfers this limit to $\rho_n$, and we conclude
    \begin{align*}
        \rho_n \stackrel{d}{\to}\dfrac{\sqrt{\alpha}v\exp(\alpha\eta)G}{\sqrt{\alpha + \beta v^2 \exp(2\alpha\eta) G^2}}
    \end{align*}Dividing the numerator and denominator by $\sqrt{\beta}$, we get
    \begin{align*}
        \rho_n \stackrel{d}{\to}\dfrac{\rho_*v\exp(\alpha\eta)G}{\sqrt{\rho_*^2 + v^2 \exp(2\alpha\eta) G^2}}
    \end{align*}Finally, using the formula for $\rho_*$, note that
    \begin{align*}
        v^2\times \rho_*^2 &= \left(1 + \dfrac{(\theta^2+1)\delta^2}{2\delta(\theta^2-\delta/2)}\right)\times \dfrac{\theta^2-\delta/2}{\theta^2(1+\delta/2)}\\
        &= \dfrac{2(\theta^2-\delta/2)+\delta(\theta^2+1)}{2(\theta^2-\delta/2)}\times \dfrac{\theta^2-\delta/2}{\theta^2(1+\delta/2)}\\
        &=\theta^2(1 + \delta/2)\times\dfrac{1}{\theta^2(1+\delta/2)} = 1
    \end{align*}Thus, $v=1/\rho_*$. Substituting this into the limit above, we complete the proof.
\end{proof}


\medskip

A simple extension of the proof of Theorem \ref{thm:critical} yields the following value of the overlap around $n=[\gamma_*d\log d + t_d d]$. Note that since $t_d\to\infty$, we have crossed the critical region and we are in the vicinity of $\rho_*$.

\medskip

\begin{lemma}\label{lemma:critical_upper_boundary}
    Let $n_+=[\gamma_*d\log d+t_d d]$. Then,
    \begin{align*}
        |\rho_{n_+}|\stackrel{p}{\to}|\rho_*|
    \end{align*}
\end{lemma}

\medskip

\begin{proof}[Proof of Lemma \ref{lemma:critical_upper_boundary}]
    This is a simple extension of the proof of Theorem \ref{thm:critical}. Since $k=2t_d d$ in this case, Lemma \ref{lemma:ode_tracking} implies that $\rho_n-\rho(2t_d)\stackrel{p}{\to}0$. Now,
    \begin{align*}
        \rho(2t_d) &= \dfrac{\sqrt{\alpha}\rho_{n_-}\exp(2\alpha t_d)}{\sqrt{\alpha + \beta \rho^2_{n_-}(\exp(4\alpha t_d)-1)}}\\
        &= \dfrac{\sqrt{\alpha}\exp(\alpha t_d)\rho_{n_-}}{\sqrt{\alpha \exp(-2\alpha t_d)+\beta\rho^2_{n_-}(\exp(2\alpha t_d)-\exp(-2\alpha t_d))}}
    \end{align*}and since $\exp(\alpha t_d)\rho_{n_-}\stackrel{d}{\to}vG$ for $G\sim \gN(0,1)$, we get $\rho(2t_d)\stackrel{d}{\to}\rho_*S$, for $S=\sign(G)$. Since $G\sim\gN(0,1)$, $S$ is Rademacher. Consequently, Slutsky's lemma implies that $\rho_n\stackrel{d}{\to}\rho_*S$, which further implies the result.
\end{proof}

\medskip

Lemma \ref{lemma:critical_upper_boundary} is important because it will help us establish that when $n/d\log(d) \to \gamma > \gamma_*$, i.e. we are in the supercritical region, the squared overlap remains $\rho_*^2$. This establishes the curious phenomenon that the limiting overlap does not depend on $\gamma$ at all!


\medskip

\begin{proof}[Proof of Theorem \ref{thm:supercritical}]
    We start with recursion \ref{eqn:recursion_final}. Let $e_k=\rho_k^2-\rho_*^2$ for all $k$, then we write
    \begin{align*}
        \E[e_k^2] &= \E\left[\left(\rho_{k-1}+\dfrac{\rho_{k-1}(\alpha-\beta\rho_{k-1}^2)}{d}+\dfrac{M_k}{d}+\dfrac{R_k}{d^2}\right)^2-\rho_*^2\right]^2\\
        &\leq \E[e^2_{k-1}] + \dfrac{4\E[\rho^2_{k-1}e_{k-1}(\alpha-\beta\rho^2_{k-1})]}{d}+\dfrac{C}{d^2}
    \end{align*}where we expanded the square, used the martingale property of $M_k$ and bounded the rest of the expectations using boundedness and finite moments of the variables involved. Note that the middle term is actually negative:
    \begin{align*}
        & &e_{k-1}(\alpha-\beta\rho^2_{k-1}) &= -\beta e_{k-1}^2\\
        &\implies &\E[\rho^2_{k-1}e_{k-1}(\alpha-\beta\rho^2_{k-1})] &= -\beta \E[\rho^2_{k-1}e_{k-1}^2]<0
    \end{align*}where we used the exact form $\rho_*^2=\alpha/\beta$. Thus, 
    \begin{align*}
        \E[e_k^2] &\leq \E[e_{k-1}^2]+C/d^2
    \end{align*}which implies, for any $l<k$,
    \begin{align*}
        \E[e_k^2] &\leq \E[e_l^2] + C(k-l)/d^2
    \end{align*}We take $k=n$ and $l=n_+$. By Lemma \ref{lemma:critical_upper_boundary}, $\rho_{n_+}^2\stackrel{p}{\to}\rho_*^2$ which implies $e_{n_+}^2\stackrel{p}{\to}0$. Since $0\leq e_{n_+}^2\leq 1$, the bounded convergence theorem implies that $\E[e_{n_+}^2]\to 0$. Further, $k-l\leq n=\gamma d\log d(1 + o(1))$, so $(k-l)/d^2\leq \gamma(\log d)(1+o(1))/d\to 0$. Thus, we get $\E[e_n^2]\to0$ and hence $e_n\stackrel{p}{\to}0$. This completes the proof.
\end{proof}

We now come to the proof of Theorem \ref{thm:BAGJ}. The proof is very similar to that of Theorems \ref{thm:subcritical}, \ref{thm:critical} and \ref{thm:supercritical} for ordinary Oja, and hence we only provide the main steps. Towards this, we identify the main recursions analogous to those in Lemma \ref{lemma:main_recursion}.

\medskip

\begin{lemma}\label{lemma:BAGJ_recursion}
    Recall that $\alpha = \delta(\theta^2-\delta/2)$ and $\beta = \delta\theta^2(1+\delta/2)$. The overlap sequence $\rho^\sph_k$ satisfies the recursion
    \begin{align}\label{eqn:BAGJ_rho_recursion}
        \rho_k^\sph &= \rho_{k-1}^\sph\left(1+\dfrac{\alpha}{d}\right) - \dfrac{\beta(\rho_{k-1}^\sph)^3}{d}+\dfrac{M_k^\sph}{d} + \dfrac{R_k^\sph}{d^2}
    \end{align}where $M_k^\sph$ is a martingale difference adapted to the filtration $\gF_k:=\{\hat v_0,X_1,\cdots, X_k\}$. Both $M_k$ and $R_k$ have bounded moments of all orders. Let, as before, $c_d=1+\alpha/d$. Unfolding recursion (\ref{eqn:BAGJ_rho_recursion}), we get,
    \begin{align}\label{eqn:BAGJ_rho_recursion_unfolded}
      \rho_n^\sph &= c_d^n\rho_0^\sph - \dfrac{\beta}{d}\sum_{k=1}^n c_d^{n-k}(\rho_{k-1}^\sph)^3 + \dfrac{1}{d}\sum_{k=1}^n c_d^{n-k}M^\sph_k + \dfrac{1}{d^2}\sum_{k=1}^n R^\sph_k
    \end{align}
\end{lemma}

\medskip

\begin{proof}[Proof of Lemma \ref{lemma:BAGJ_recursion}]
    We start with the spherical Oja algorithm (\ref{eqn:spherical_oja}). We define the analogous variables as we had done for ordinary Oja.
    \begin{align*}
        \rho^\sph_k &= \langle \hat v^\sph_k,v_0\rangle, \\
        A_k &= \langle X_k, v_0\rangle, \\
        e_k^\sph &= \dfrac{\hat v_k^\sph - \rho_k^\sph v_0}{\sqrt{1-(\rho_k^\sph)^2}},\\
        C_k^\sph &= \langle X_k, e_{k-1}^\sph\rangle,\\
        B_k^\sph &= A_kv_0 + C_k^\sph e_{k-1}^\sph,
    \end{align*}
    After some algebra, we get
    \begin{align*}
        \rho^\sph_k &= \dfrac{\rho^\sph_{k-1}+(\delta/d)B_k^\sph(A_k-B_k^\sph\rho^\sph_{k-1})}{[1+(\delta^2/d^2)(B_k^\sph)^2(\|X_k\|^2-(B_k^\sph)^2)]^{1/2}}
    \end{align*}
    Then, following the approximation argument presented in the proof of Lemma \ref{eq:recursion_rho}, we can derive that
    \begin{align*}
        \rho_k^\sph &= \rho_{k-1}^\sph - \dfrac{\delta^2}{2d}(\rho^\sph_{k-1})^3A^2_k - \dfrac{\delta^2(\rho^\sph_{k-1})^2\sqrt{1-(\rho_{k-1}^\sph)^2}A_kC_k^\sph}{d}\\
        &- \dfrac{\delta^2}{2d}\rho_{k-1}^\sph(1-(\rho^\sph_{k-1})^2)(C^\sph_k)^2+\dfrac{\delta\rho_{k-1}^\sph (1-(\rho^\sph_{k-1})^2)A_k^2}{d} \\
        &- \dfrac{\delta\rho_{k-1}^\sph(1-(\rho_{k-1}^\sph)^2)(C_k^\sph)^2}{d} + \delta(1-2(\rho^\sph_{k-1})^2)\sqrt{1-(\rho^\sph_{k-1})^2}A_kC_k^\sph +\dfrac{R_k^\sph}{d^2}
    \end{align*}where $R_k^\sph$ has uniformly bounded moments of all orders. To get the martingale term, we compute the conditional mean of the right side excluding $R_k^\sph$ given $\gF_{k-1}$. Note that, just as before, $A_k,C_k$ are independent $\gN(0,\theta^2+1)$ and $\gN(0,1)$ given $\gF_{k-1}$, and hence independent of $\gF_{k-1}$. Thus, define
    \begin{align*}
        M_k^\sph &= -\dfrac{\delta^2(\rho_{k-1}^\sph)^3(A_k^2-(\theta^2+1))}{2} - \delta^2(\rho^\sph_{k-1})^2\sqrt{1-(\rho^\sph_{k-1})^2}A_kC_k^\sph\\
        &-\dfrac{\delta^2\rho_{k-1}^\sph(1-(\rho^\sph_{k-1})^2)((C^\sph_k)^2-1)}{2}+\delta\rho^\sph_{k-1}(1-(\rho^\sph_{k-1})^2)(A_k^2-(\theta^2+1)) \\
        &- \delta\rho_{k-1}^\sph(1-(\rho^\sph_{k-1})^2)((C^\sph_k)^2-1)+ \delta(1-2(\rho^\sph_{k-1})^2)\sqrt{1-(\rho^\sph_{k-1})^2}A_kC_k^\sph
    \end{align*}which ultimately gives, after some algebra, the desired recursion (\ref{eqn:BAGJ_rho_recursion}). Recursion (\ref{eqn:BAGJ_rho_recursion_unfolded}) is a straightforward corollary.
\end{proof}

\medskip

\begin{proof}[Proof of Theorem \ref{thm:BAGJ}]
    The martingales $M_k$ from Lemma \ref{lemma:main_recursion} and $M_k^\sph$ from Lemma \ref{lemma:BAGJ_recursion} are visibly different. However, recalling (\ref{ineq:Sk_bound}), in both cases, we have the same leading order term:
\begin{align*}
    M_k &= \delta A_k C_k + \rho_k \times (\text{terms with bounded moments})\\
    M_k^\sph &= \delta A_k C_k^\sph + \rho^\sph_k \times (\text{terms with bounded moments})
\end{align*}Consequently, we can derive an equivalent of Lemma \ref{lemma:martingale_simplified} by defining $\tilde M_k^\sph = \delta A_kC_k^\sph$ and $S_k^\sph=M_k^\sph-\tilde M_k^\sph$. This ensures that all the asymptotic characterizations are exactly identical in the two cases. Thereafter, the exact proof methodology as in ordinary Oja works for spherical Oja, and the exact same conclusion can be derived. This proves Theorem \ref{thm:BAGJ}.
\end{proof}

\paragraph{Acknowledgements.} Part of this work was conducted while the author was at Stanford University, hosted by David Donoho. The computational experiments presented in this paper were all done on Stanford's Sherlock compute cluster.

\bibliographystyle{conference}  



\begin{thebibliography}{35}
\providecommand{\natexlab}[1]{#1}
\providecommand{\url}[1]{\texttt{#1}}
\expandafter\ifx\csname urlstyle\endcsname\relax
  \providecommand{\doi}[1]{doi: #1}\else
  \providecommand{\doi}{doi: \begingroup \urlstyle{rm}\Url}\fi

\bibitem[Amini \& Wainwright(2008)Amini and Wainwright]{amini2008high}
Arash~A Amini and Martin~J Wainwright.
\newblock High-dimensional analysis of semidefinite relaxations for sparse principal components.
\newblock In \emph{2008 IEEE international symposium on information theory}, pp.\  2454--2458. IEEE, 2008.

\bibitem[Baik et~al.(2005)Baik, Arous, and P{\'e}ch{\'e}]{baikbenarouspeche}
Jinho Baik, G{\'e}rard~Ben Arous, and Sandrine P{\'e}ch{\'e}.
\newblock {Phase transition of the largest eigenvalue for nonnull complex sample covariance matrices}.
\newblock \emph{The Annals of Probability}, 33\penalty0 (5):\penalty0 1643 -- 1697, 2005.
\newblock \doi{10.1214/009117905000000233}.
\newblock URL \url{https://doi.org/10.1214/009117905000000233}.

\bibitem[Ben~Arous et~al.(2021)Ben~Arous, Gheissari, and Jagannath]{arous2021online}
Gerard Ben~Arous, Reza Gheissari, and Aukosh Jagannath.
\newblock Online stochastic gradient descent on non-convex losses from high-dimensional inference.
\newblock \emph{Journal of Machine Learning Research}, 22\penalty0 (106):\penalty0 1--51, 2021.

\bibitem[Ben~Arous et~al.(2022)Ben~Arous, Gheissari, and Jagannath]{ben2022high}
Gerard Ben~Arous, Reza Gheissari, and Aukosh Jagannath.
\newblock High-dimensional limit theorems for sgd: Effective dynamics and critical scaling.
\newblock \emph{Advances in neural information processing systems}, 35:\penalty0 25349--25362, 2022.

\bibitem[Benaych-Georges \& Nadakuditi(2011)Benaych-Georges and Nadakuditi]{benaych2011eigenvalues}
Florent Benaych-Georges and Raj~Rao Nadakuditi.
\newblock The eigenvalues and eigenvectors of finite, low rank perturbations of large random matrices.
\newblock \emph{Advances in Mathematics}, 227\penalty0 (1):\penalty0 494--521, 2011.

\bibitem[Benaych-Georges \& Nadakuditi(2012)Benaych-Georges and Nadakuditi]{benaych2012singular}
Florent Benaych-Georges and Raj~Rao Nadakuditi.
\newblock The singular values and vectors of low rank perturbations of large rectangular random matrices.
\newblock \emph{Journal of Multivariate Analysis}, 111:\penalty0 120--135, 2012.

\bibitem[Billingsley(2017)]{billingsley2017probability}
Patrick Billingsley.
\newblock \emph{Probability and measure}.
\newblock John Wiley \& Sons, 2017.

\bibitem[Boutsidis et~al.(2014)Boutsidis, Garber, Karnin, and Liberty]{boutsidis2014online}
Christos Boutsidis, Dan Garber, Zohar Karnin, and Edo Liberty.
\newblock Online principal components analysis.
\newblock In \emph{Proceedings of the twenty-sixth annual ACM-SIAM symposium on Discrete algorithms}, pp.\  887--901. SIAM, 2014.

\bibitem[Cardot \& Degras(2018)Cardot and Degras]{cardot2018online}
Herv{\'e} Cardot and David Degras.
\newblock Online principal component analysis in high dimension: Which algorithm to choose?
\newblock \emph{International Statistical Review}, 86\penalty0 (1):\penalty0 29--50, 2018.

\bibitem[Deshpande \& Montanari(2014)Deshpande and Montanari]{deshpande2014information}
Yash Deshpande and Andrea Montanari.
\newblock Information-theoretically optimal sparse pca.
\newblock In \emph{2014 IEEE International Symposium on Information Theory}, pp.\  2197--2201. IEEE, 2014.

\bibitem[Deshpande \& Montanari(2016)Deshpande and Montanari]{deshpande2016sparse}
Yash Deshpande and Andrea Montanari.
\newblock Sparse pca via covariance thresholding.
\newblock \emph{Journal of Machine Learning Research}, 17\penalty0 (141):\penalty0 1--41, 2016.

\bibitem[Donoho \& Tanner(2009)Donoho and Tanner]{donoho2009observed}
David Donoho and Jared Tanner.
\newblock Observed universality of phase transitions in high-dimensional geometry, with implications for modern data analysis and signal processing.
\newblock \emph{Philosophical Transactions of the Royal Society A: Mathematical, Physical and Engineering Sciences}, 367\penalty0 (1906):\penalty0 4273--4293, 2009.

\bibitem[Gheissari \& Jagannath(2025)Gheissari and Jagannath]{gheissari2025universality}
Reza Gheissari and Aukosh Jagannath.
\newblock Universality of high-dimensional scaling limits of stochastic gradient descent.
\newblock \emph{arXiv preprint arXiv:2512.13634}, 2025.

\bibitem[Henriksen \& Ward(2019)Henriksen and Ward]{henriksen2019adaoja}
Amelia Henriksen and Rachel Ward.
\newblock Adaoja: Adaptive learning rates for streaming pca.
\newblock \emph{arXiv preprint arXiv:1905.12115}, 2019.

\bibitem[Jain et~al.(2016)Jain, Jin, Kakade, Netrapalli, and Sidford]{jain2016streaming}
Prateek Jain, Chi Jin, Sham~M Kakade, Praneeth Netrapalli, and Aaron Sidford.
\newblock Streaming pca: Matching matrix bernstein and near-optimal finite sample guarantees for oja’s algorithm.
\newblock In \emph{Conference on learning theory}, pp.\  1147--1164. PMLR, 2016.

\bibitem[Johnstone(2001)]{johnstone2001distribution}
Iain~M Johnstone.
\newblock On the distribution of the largest eigenvalue in principal components analysis.
\newblock \emph{The Annals of statistics}, 29\penalty0 (2):\penalty0 295--327, 2001.

\bibitem[Johnstone \& Lu(2009)Johnstone and Lu]{johnstone2009consistency}
Iain~M Johnstone and Arthur~Yu Lu.
\newblock On consistency and sparsity for principal components analysis in high dimensions.
\newblock \emph{Journal of the American Statistical Association}, 104\penalty0 (486):\penalty0 682--693, 2009.

\bibitem[Johnstone \& Paul(2018)Johnstone and Paul]{johnstone2018pca}
Iain~M Johnstone and Debashis Paul.
\newblock Pca in high dimensions: An orientation.
\newblock \emph{Proceedings of the IEEE}, 106\penalty0 (8):\penalty0 1277--1292, 2018.

\bibitem[Krasulina(1970)]{krasulina1970method}
T~Krasulina.
\newblock Method of stochastic approximation in the determination of the largest eigenvalue of the mathematical expectation of random matrices.
\newblock \emph{Automatation and remote control}, 2:\penalty0 50--56, 1970.

\bibitem[Kumar \& Sarkar(2023)Kumar and Sarkar]{kumar2023streaming}
Syamantak Kumar and Purnamrita Sarkar.
\newblock Streaming pca for markovian data.
\newblock \emph{Advances in Neural Information Processing Systems}, 36:\penalty0 64650--64662, 2023.

\bibitem[Kumar et~al.(2025)Kumar, Pandey, and Sarkar]{kumar2025beyond}
Syamantak Kumar, Shourya Pandey, and Purnamrita Sarkar.
\newblock Beyond sin-squared error: Linear-time entrywise uncertainty quantification for streaming pca.
\newblock \emph{arXiv preprint arXiv:2506.12655}, 2025.

\bibitem[Li et~al.(2017)Li, Wang, Liu, and Zhang]{li2017diffusion}
Chris~Junchi Li, Mengdi Wang, Han Liu, and Tong Zhang.
\newblock Diffusion approximations for online principal component estimation and global convergence.
\newblock \emph{Advances in Neural Information Processing Systems}, 30, 2017.

\bibitem[Li et~al.(2023)Li, Fan, and Wei]{LiFanWei}
Gen Li, Wei Fan, and Yuting Wei.
\newblock Approximate message passing from random initialization with applications to <i>z</i><sub>2</sub> synchronization.
\newblock \emph{Proceedings of the National Academy of Sciences}, 120\penalty0 (31):\penalty0 e2302930120, 2023.
\newblock \doi{10.1073/pnas.2302930120}.
\newblock URL \url{https://www.pnas.org/doi/abs/10.1073/pnas.2302930120}.

\bibitem[Lunde et~al.(2021)Lunde, Sarkar, and Ward]{lunde2021bootstrapping}
Robert Lunde, Purnamrita Sarkar, and Rachel Ward.
\newblock Bootstrapping the error of oja's algorithm.
\newblock \emph{Advances in neural information processing systems}, 34:\penalty0 6240--6252, 2021.

\bibitem[Mardia et~al.(2024)Mardia, Kent, and Taylor]{mardia2024multivariate}
Kanti~V Mardia, John~T Kent, and Charles~C Taylor.
\newblock \emph{Multivariate analysis}.
\newblock John Wiley \& Sons, 2024.

\bibitem[Montanari \& Venkataramanan(2021)Montanari and Venkataramanan]{MontanariVenkataraman}
Andrea Montanari and Ramji Venkataramanan.
\newblock {Estimation of low-rank matrices via approximate message passing}.
\newblock \emph{The Annals of Statistics}, 49\penalty0 (1):\penalty0 321 -- 345, 2021.
\newblock \doi{10.1214/20-AOS1958}.
\newblock URL \url{https://doi.org/10.1214/20-AOS1958}.

\bibitem[Nie et~al.(2016)Nie, Kotlowski, and Warmuth]{nie2016online}
Jiazhong Nie, Wojciech Kotlowski, and Manfred~K Warmuth.
\newblock Online pca with optimal regret.
\newblock \emph{Journal of Machine Learning Research}, 17\penalty0 (173):\penalty0 1--49, 2016.

\bibitem[Oja(1982)]{oja1982simplified}
Erkki Oja.
\newblock Simplified neuron model as a principal component analyzer.
\newblock \emph{Journal of mathematical biology}, 15\penalty0 (3):\penalty0 267--273, 1982.

\bibitem[Oja \& Karhunen(1985)Oja and Karhunen]{oja1985stochastic}
Erkki Oja and Juha Karhunen.
\newblock On stochastic approximation of the eigenvectors and eigenvalues of the expectation of a random matrix.
\newblock \emph{Journal of mathematical analysis and applications}, 106\penalty0 (1):\penalty0 69--84, 1985.

\bibitem[Perry et~al.(2018)Perry, Wein, Bandeira, and Moitra]{perry2018optimality}
Amelia Perry, Alexander~S Wein, Afonso~S Bandeira, and Ankur Moitra.
\newblock Optimality and sub-optimality of pca i: Spiked random matrix models.
\newblock \emph{The Annals of Statistics}, 46\penalty0 (5):\penalty0 2416--2451, 2018.

\bibitem[Pham et~al.(2025)Pham, Rinaldo, and Sarkar]{pham2025time}
Tuan Pham, Alessandro Rinaldo, and Purnamrita Sarkar.
\newblock Time-uniform concentration bounds for iterative algorithms.
\newblock \emph{arXiv preprint arXiv:2511.18273}, 2025.

\bibitem[Robinson(2004)]{robinson2004introduction}
James~C Robinson.
\newblock \emph{An introduction to ordinary differential equations}.
\newblock Cambridge University Press, 2004.

\bibitem[Wang \& Lu(2016)Wang and Lu]{wang2016online}
Chuang Wang and Yue~M Lu.
\newblock Online learning for sparse pca in high dimensions: Exact dynamics and phase transitions.
\newblock In \emph{2016 IEEE Information Theory Workshop (ITW)}, pp.\  186--190. IEEE, 2016.

\bibitem[Wang et~al.(2017)Wang, Mattingly, and Lu]{wang2017scaling}
Chuang Wang, Jonathan Mattingly, and Yue~M Lu.
\newblock Scaling limit: Exact and tractable analysis of online learning algorithms with applications to regularized regression and pca.
\newblock \emph{arXiv preprint arXiv:1712.04332}, 2017.

\bibitem[Warmuth \& Kuzmin(2008)Warmuth and Kuzmin]{warmuth2008randomized}
Manfred~K Warmuth and Dima Kuzmin.
\newblock Randomized online pca algorithms with regret bounds that are logarithmic in the dimension.
\newblock \emph{Journal of Machine Learning Research}, 9\penalty0 (10):\penalty0 2287--2320, 2008.

\end{thebibliography}



\end{document}